\documentclass[a4paper,12pt]{article} 

\usepackage[english]{babel}
\usepackage{amsfonts,amsmath,amsxtra,amsthm,latexsym}
\usepackage{amssymb} % \Subset !

\usepackage{graphicx}\usepackage{graphicx}

\newtheorem{thm}{Theorem}[section]
 \newtheorem{cor}[thm]{Corollary}
 \newtheorem{lem}[thm]{Lemma}
 \newtheorem{prop}[thm]{Proposition}
 \theoremstyle{definition}
 \newtheorem{defn}{Definition}[section]
 \theoremstyle{remark}
 \newtheorem{rem}{Remark}[section]
   \newtheorem{ex}[thm]{Example}
  \newtheorem{conjecture}{Conjecture}

 \numberwithin{equation}{section}

\DeclareMathOperator{\im}{Im}
\DeclareMathOperator{\re}{Re}

\DeclareMathOperator{\cone}{cone}

\DeclareMathOperator{\Ext}{ext}
\DeclareMathOperator{\Int}{int}
\DeclareMathOperator{\supp}{supp}
\DeclareMathOperator{\meas}{meas}
\DeclareMathOperator{\sgn}{sgn}
\DeclareMathOperator{\Bd}{bd}
\DeclareMathOperator{\Arctan}{Arctan}
\DeclareMathOperator*{\esssup}{ess\,sup}
\DeclareMathOperator*{\essinf}{ess\,inf}

\DeclareMathOperator{\En}{En}
\DeclareMathOperator{\Dr}{Dr}
\DeclareMathOperator{\pr}{pr}

\def\RR{\mathbb R}
\def\CC{\mathbb C}
\def\NN{\mathbb N}
\def\ZZ{\mathbb Z}
\def\DD{\mathbb D}

\def\BB{\mathbb B}

\def\A{\mathbb{A}}

\def\vphi{\varphi}

\def\de{\delta}
\def\ep{\varepsilon}

\def\pa{\partial}
\def\om{\omega}
\def\Si{\Sigma}
\def\Om{\Omega}
\def\La{\Lambda}

\def\wt{\widetilde}

\def\ii{\mathrm{i}}
\def\dd{\mathrm{d}}
\def\nl{\mathrm{nl}}
\def\D{\mathcal{D}}
\def\B{\mathcal{B}}
\def\Z{\mathcal{Z}}
\def\M{\mathcal{M}}

\def\ChiCpl{\chi_{_{\scriptstyle \CC_+}}}

\def\Setim{\Sigma_{\mathrm{im}}}
\def\mult{\mathrm{mult}}
\def\Mes{\mathbb{M}}
\def\modn{\hspace{-10pt}\mod}
\def\wlim{\mathrm{w}^* \text{-}  \lim}
\def\Sset{\mathcal{S}}

\begin{document}
\title{Nonlinear bang-bang eigenproblems and optimization of resonances in layered cavities}
\author{} 
\date{}
\maketitle
{\center 
{\large 
Illya M. Karabash $^{\text{a,b,c,*}}$, Olga M. Logachova $^{\text{d}}$, Ievgen V. Verbytskyi $^{\text{e}}$ 
\\[4mm]
}} 

{\small \noindent
$^{\text{a}}$
Institut f{\"u}r Mathematik, Universit{\"a}t zu L{\"u}beck, 
 Haus 64, 3. OG Ratzeburger Allee 160, D-23562 L{\"u}beck, Germany\\[1mm]
 $^{\text{b}}$ Institute of Applied Mathematics and Mechanics of NAS of Ukraine,
Dobrovolskogo st. 1, Slovyans'k 84100, Ukraine\\[1mm]
$^{\text{c}}$  Donetsk National University, Department of Mathematical Analysis and Differential Equations, 
Faculty of Mathematics and Information Technology,  600-richya st. 21, Vinnytsia 21021, Ukraine\\[1mm]
$^{\text{d}}$  
Siberian State University of Geosystems and Technologies, Department of Higher Mathematics, Institute of
Geodesy and Management, Plakhotnogo st. 10, Novosibirsk 630108, Russia \\[1mm]
$^{\text{e}}$ National Technical University of Ukraine "Kyiv Polytechnic Institute", 
Department of Industrial Electronics, Faculty of Electronics,  Politekhnichna st. 16, block 12, Kyiv 03056, Ukraine\\[1mm] 
$^{\text{*}}$ Corresponding author. E-mail: i.m.karabash@gmail.com\\[1cm]
}

\begin{abstract}
Quasi-normal-eigenvalue optimization is studied under constraints
$b_1(x) \le B(x) \le b_2 (x)$ on structure functions $B$ of 2-side open optical or mechanical 
resonators. We prove existence of various optimizers and provide an example when different structures generate the same optimal quasi-(normal-)eigenvalue.
To show that quasi-eigenvalues  locally optimal in various senses 
are in the spectrum $\Si^\nl$ of 
the bang-bang eigenproblem $y''  = - \om^2 y  [ b_1 + (b_2 - b_1) \ChiCpl (y^2 ) ]$, where
$\ChiCpl (\cdot)$ is the indicator function of $\CC_+$,  we obtain 
a variational characterization of  $\Si^\nl$ in terms of quasi-eigenvalue perturbations. To address the  minimization of the decay rate $| \im \om |$, 
we study the bang-bang equation and explain how it excludes an unknown optimal $B$ from the optimization process.  Computing one of minimal decay structures for 1-side open settings, we show that it resembles gradually size-modulated 1-D stack cavities introduced recently in Optical Engineering. In 2-side open symmetric settings, our example has an additional centered defect. Nonexistence 
of global decay rate minimizers is discussed.
\end{abstract}

\quad\\
MSC-classes: 49R05, %Variational methods for eigenvalues of operators
47J10, % Nonlinear spectral theory, nonlinear eigenvalue problems 
35B34, % Partial differential equations, Qualitative properties of solutions, Resonances
49M05, % Numerical methods: 	Methods based on necessary conditions 
46N10, % Functional analysis: Miscellaneous applications of functional analysis: 
%  	Applications in optimization, convex analysis, mathematical programming, economics
90C29, %   	Multi-objective and goal programming
47A55,  %Operator theory:   	Perturbation theory
78M50,  % Optics, electromagnetic theory: Optimization
70J10,  % Mechanics of particles and systems: Linear vibration theory: Modal analysis
34L15
% Ordinary differential equations: Eigenvalues, estimation of eigenvalues, upper and lower bounds

\tableofcontents

\pagebreak 

\section{Introduction}
\label{s Intro}

The paper is devoted to the analytical and numerical study of two connected questions: 
nonlinear eigenproblems with a bang-bang term and optimization of quasi-(normal-)eigenvalues 
in optical and mechanical resonators.

A leaky cavity (or resonator) is a region of space within which the
electromagnetic (or acoustic) field is well confined, but not completely
confined. Cavities supporting eigenmodes with high quality factor (high-Q) are needed for 
a number of applications including Cavity Quantum Electrodynamics, Optical Engineering, 
and Microscopy  (see e.g. \cite{VLMS02,PDBRS02,PMS02,Sch11}).
The recent  progress in fabrication of small size optical resonators  
\cite{AASN03, KTTKRN10, LJ13} attracted considerable interest to numerical 
\cite{KS08,HBKW08,BPSchZsch11,GCh13,LJ13,MPBKLR13} and analytical   \cite{Ka13,Ka13_KN,Ka14} aspects of resonance optimization.

Optical  quasi-eigenvalues (or resonances) can be mathematically defined as 
eigenvalues $\om$ corresponding to outgoing solutions of 
the time-harmonic Maxwell equation  \cite{S68,LJ13}. 
Because of the leakage, resonant eigenoscillations are expected to decay exponentially
in time. The modulus of the imaginary part $|\im \om|$ of a quasi-eigenvalue corresponds to \emph{the decay rate} of the standing wave,
the real part $\re \om$ to \emph{the frequency of oscillations}. 
In Engineering studies, the confinement of energy for a particular resonant eigenmode is measured by \emph{the quality factor} $Q (\om) = \frac{\re \om}{-2 \im \omega}$.

Since light is essentially difficult to localize it is hard to realize small-sized optical 
cavities with strong light confinement.  
One of the reasons of active involvement of Numerics in this field is that 
the initial progress with fabrication of high-Q cavities based on 2-D photonic crystals 
was achieved with  the help of computer simulations of eigenmodes  
for perspective versions of design \cite{AASN03}. 
The interest  to multilayer structures with 1-D geometry has soon 
returned \cite{NKT08, MHvG08, KTTKRN10, BPSchZsch11,MPBKLR13},  
partially because of cheaper fabrication process, partially because computations 
in this case are simpler.  For idealized 1-D cavities with infinite 
layers and under the assumption of normally passing electromagnetic (EM) waves,   the Maxwell system can be reduced to a 1+1  wave equation of nonhomogeneous string $B(x) \pa_t^2 u (x,t) = \pa_x^2 u (x,t) $ (see e.g. \cite{U75,Sch11}). The analogy with a string is widely used in Optical Engineering \cite{LSL73,MvdBY01} (and possibly goes back to Lord Rayleigh \cite[Chapter 4]{JJWM08}).
 
Quasi-eigenvalues associated with $B(x) \pa_t^2 u (x,t) = \pa_x^2 u (x,t) $ can be defined via the eigenproblem consisting of  the time-harmonic equation 
\begin{eqnarray}
y''(x) & = & -\om^2 B(x) y (x)  \qquad \text{a.e. on $ a_1 < x < a_2 $,} 
% \quad $a_{1,2} \in \RR$},
 \label{e ep} 
\end{eqnarray}
equipped with radiation boundary conditions.
In this 1-D case, the radiation conditions take the $\om$-dependent form of local conditions at the endpoints 
\begin{eqnarray}
\frac{y' (a_1)}{-i\om} & = & \nu_1 y(a_1), \label{e bc a-} \\
\frac{y' (a_2)}{i\om\nu_2} & = & y(a_2) \qquad \text{(see Section \ref{s basics} for details)}. \label{e bc a+}
\end{eqnarray}

The following properties of the \emph{set $\Sigma (B)$ of quasi-eigenvalues associated with the coefficient $B(\cdot)$} are important for the
present paper: $\Sigma (B)$ is a subset of the lower complex half-plane $\CC_- = \{z \in \CC : \im z < 0 \}$ and is symmetric with respect to
(w.r.t.) the imaginary axis $\ii \RR$, quasi-eigenvalues are isolated and of finite algebraic multiplicity, $\infty$ is their only
possible accumulation point (see Section \ref{s mult and analyt} and, for one-side open cases,
\cite{KN79,KN89,CZ95,GP97,PvdM01}).

From mathematical point of view, estimates on resonances associated with various wave equations have being   
studied in Mathematical Physics at least since 1970s \cite{LP71} and is still an active area of 
research (see e.g. \cite{AT11} and references therein). Optimization of resonances may be seen as an attempt to obtain 
sharp estimates of this kind. This point of view  and the study of resonances associated with 
random Schr\"odinger operators were initial sources of 
the interest to the problem \cite{H82,HS86}. It seems that, up to now, 
sharp estimates have been obtained only in low frequency regions for 
a few 1-D models involving total mass type constraints \cite{Ka13_KN,Ka14} (on a somewhat 
different problem involving damping, see \cite{CZ94} 
and the discussion in \cite{CO96}). 
 Among other variational problems for eigenvalues,
the optimization of resonances can be classified as the nonselfadjoint spectral optimization.
It contains essentially new effects and difficulties in comparison with selfadjoint cases (see \cite{CZ95,CO96,BLO01,Ka13,Ka14}).

Since quasi-eigenvalues $\om$ are complex numbers, one of the ways 
to formulate related variational problems is to interpret $\om$ as an $\RR^2$-vector 
and to take the point of view of the optimization theory for vector-valued cost functions 
(Pareto optimization) \cite{Ka14}. 
The approach of  \cite{Ka14}  is a development of that of \cite{Ka13}, 
where the engineering problem of the high-Q design for one-side open multilayer cavities was considered analytically. 
To provide a simple and simultaneously rigorous formulation, 
the papers \cite{Ka13_KN, Ka13} restate the problem of 
high-Q design in terms of the minimization of the decay rate $|\im \om|$ 
for quasi-eigenvalues with a fixed frequency $\re \om = \alpha$. 
Then the set of all \emph{quasi-eigenvalues with minimal possible decay for their particular frequencies}
forms a \emph{Pareto optimal frontier}  (see Fig.~\ref{f drawing} in Section \ref{s opt problem} and  \cite[Fig.~1]{Ka14}).  

The quasi-eigenvalues are supposed to be generated by cavities with structures belonging to a certain admissible family. 
In the present paper, the resonator's \emph{structure} is described by 
an $L^1$-function $B(x)$ on $[a_1,a_2]$ ($-\infty<a_1<a_2 < +\infty$), which represents 
spatially varying dielectric permittivities in the case of optical cavity, or the linear density 
in Mechanics settings. Boundary conditions at the endpoints $a_{1,2}$, 
which describe the leakage of energy due to radiation or damping, 
include two parameters $\nu_1, \nu_2 \in [0,+\infty]$. 
These parameters are assumed to be fixed, and so, they do not participate in 
the optimization process. The admissible family of the structures $\A$ 
(over which the optimization is performed) is defined by the functional 
side constraints $b_1(x) \le B(x) \le b_2 (x)$ a.e. on $(a_1,a_2)$. 
This allows us to consider an important for applications situation when the 
resonator itself is only a part of a more complex device \cite{MHvG08,LJ13}.  
That is, sometimes only some parts of the system are  suitable  for modifications to achieve better resonant properties. The part of the device that is not subjected to optimization corresponds to the intervals where $b_1 (x) = b_2 (x)$. 

The connection between quasi-eigenvalues of minimal decay and nonlinear 
eigenproblems was established recently in \cite{Ka13,Ka13_Opt}
for the case of 1-D models of one-side open resonators and 
constant side constraints $b_1$ and $ b_2 $.
It was proved that the set of Pareto optimal  quasi-eigenvalues of minimal decay belongs to the spectrum associated with the    
equation 
%\begin{equation} \label{e nl eq int}
$y'' (x)  = - \om^2 y(x)  \left[ b_1 + [b_2 - b_1] \ChiCpl (y^2 (x) ) \right]$.
%\end{equation}
The nonlinearity contains the function of bang-bang type 
$\B (y) (x) = \left[ b_1 + (b_2 - b_1) \ChiCpl (y^2 ) \right]$ taken of the square of 
complex-valued solution $y$. Here and below $\ChiCpl (\cdot)$ is the indicator (characteristic)
function 
 of the open upper complex half-plane $\CC_+ = \{ z : \im z > 0\}$,
i.e., 
\[
\text{$\ChiCpl (z) = 1$ if $z \in \CC_+$, and $\ChiCpl (z) = 0$ if $z \in \overline{\CC_-} = \{z : \im z \le 0 \}$. } 
\] 
So $\B (y) ( \cdot) $ equals $b_1 (\cdot)$ on the $x$-intervals  where $\im y^2 (x) \le 0 $, and $\B (y) ( \cdot) = b_2 (\cdot)$ for $x$-intervals where $\im y^2 (x) > 0$.
This resembles to some extend bang-bang equations arising in the optimal control theory when the control is restricted by a lower and an upper bound. However, the interpretation of the above equation 
from the  point of view of control is not found yet.

Since the nonlinear equation derived in \cite{Ka13} excludes infinite-dimensional unknown structure $B$, 
it may essentially simplify computation of 
optimal resonances in comparison with step-by-step adjustment procedures for $B$ based on computation 
of electromagnetic field on each step. 

Development and analytic justification of a numerical procedure based on the bang-bang eigenproblem
is one of the goals of our research.
However, the implementation of this new computational approach requires better 
understanding of the shapes of the nonlinear spectrum 
and of the Pareto optimal set, as well as their projections to the real line (see Section \ref{ss ex A}). 
Our method also requires the study of properties of solutions of the bang-bang equation (\ref{e nonlin i}), 
which are considered in Section 
\ref{ss Sinl=0} and Appendix \ref{ss uni nl}.

We show that not only quasi-eigenvalues of minimal decay, but also quasi-eigenvalues locally extremal in a much wider sense  belong to the set $\Si^\nl$ of eigenvalues associated with the nonlinear equation
\begin{equation} \label{e nonlin i}
y'' (x)  = - \om^2 y  \left[ b_1 (x) + [b_2(x) - b_1(x)] \ChiCpl (y^2 (x) ) \right] , \quad a_1 <x<a_2 
\end{equation}
(which includes the functional side constraints $b_{1,2} (\cdot)$ and is supposed to be equipped with the radiation boundary conditions). 
In particular, we show that \emph{nonlinear spectrum} $\Si^\nl$ contains the 
nonimaginary part $\Bd \Si [\A] \setminus \ii \RR$ of the boundary $\Bd \Si [\A]$ of 
the \emph{set of admissible quasi-eigenvalues} 
\[
\Si [\A] := \bigcup_{B \in \A} \Si (B) \qquad \text{ (see Sections \ref{s basics} and \ref{s Opt=>nl}).}
\] 
 To achieve this, we provide a variational characterization of 
the part of $\Si^\nl$ lying in $\CC \setminus \ii \RR$ (Theorem \ref{t char Si nl}). 
We also show that $\Bd \Si [\A]$ and so $\Si^\nl$ are, at least in some cases, 
essentially `greater' than  the set of quasi-eigenvalues of minimal decay 
(see Sections \ref{s small cont} and \ref{s NumEx} and Example \ref{ex interior}).

The quasi-eigenvalue of minimal decay for a particular frequency $\alpha$ 
can be found as the closest to $\RR$ point of the intersection of 
the nonlinear spectrum with the line $\alpha + \ii \RR$  (see Section \ref{s finding}). 
Then, corresponding optimal structures $B$ can also be found from the bang-bang eigenproblem. 
In the end of the paper, we support our study by a numerical experiment (see Fig. \ref{f Snl} (a) and Tables 1-2 in Section \ref{s NumEx}).

To prove rigorously that points of the nonlinear spectrum are not necessarily quasi-eigenvalues 
of minimal decay, the case of small dielectric contrast is considered in Section \ref{s small cont}.
In Optics literature the study of small contrast is often serves as a base of intuitive understanding 
of qualitative effects (see e.g. \cite[Chapter 4]{JJWM08}). 
In the optimization settings, small contrast 
means that the constraints $b_1$ and $b_2$ are close to each other. 
To measure the corresponding distance we use the $L^1$-norm $\|b_2 - b_1 \|_1$ since this 
choice is flexible enough to cover both the small changes in permittivity function 
(when $\|b_2 - b_1\|_\infty$ is small) and  fluctuations in the widths and positions of layers.  
This gives a connection with  models involving uncertainties or random deviations in the resonator's 
structure 
$B(x)$. 
Indeed, for stochastic models, all possible positions of random resonances (the supports of probabilistic distributions) can be localized by the set of admissible resonances $\Si [\A]$ of a properly posed optimization problem. In practical situations uncertainties or fluctuations in the cavity's structure may model  unintended deviations in fabrication process \cite{WMTV09}, incorporation of biomolecular thin films in X-ray multi-layer cavities \cite{PMS02}, and various Physics effects affecting EM properties of materials \cite{JShLV00}. 

Pure imaginary quasi-eigenvalues (i.e., the case $\om \in \ii \RR$)  correspond to 
the over-damping and critical damping effects (in Mechanics settings). 
For them, optimization partially resembles that for selfadjoint problems \cite{CZ94,CZ95,CO96,Ka13}. 
The pure imaginary case is considered separately in Section \ref{s 0fr}. It provides, as a by-product, an example of two different optimal structures associated with one quasi-eigenvalue of minimal decay. 
This partially answers  the uniqueness of optimizer question discussed in \cite{HS86,Ka13_KN,Ka14}.

The application of steepest ascent numerical procedures to 
the search of quasi-eigenvalues with locally minimal decay rate (or  locally maximal quality factor) 
attracted a considerable attention \cite{HBKW08,KS08,Sch11}. It seems that Corollary \ref{c const b12 0 adm} (i) and 
Theorem \ref{t small diel contrast} provide 
\emph{first rigorous proofs of existence of such local minimizers}. 
(The proof of  Theorem \ref{t small diel contrast} can be easily transformed to consider 
local maximizers for the Q-factor). 
\emph{The existence of global minimizer for decay rate} 
(over all $B \in \A$ and $\om \in \CC_-$ such that $\om \in \Si (B)$) \emph{is questionable}. Our opinion is that it is difficult to expect such existence in most of practical situations, see the discussion in Section \ref{ss GlobMin}. 

The discussion section (Section \ref{s dis}) also compares the results 
of the numerical experiment of Section \ref{s NumEx} with optimal 
design suggestions of engineering and numerical optimization papers. 
Our conclusion is that the gradually size-modulated 
1-D stack designs of \cite{NKT08,KTTKRN10,BPSchZsch11} are reasonable, 
but for symmetric resonators, it is possible that \emph{an additional defect in the center may 
improve the confinement of energy.}

 The results of the paper were partially reported at the International Congress of Mathematicians 
 (Seoul ICM 2014) and in several seminar and workshop talks at the University of L{\"{u}}beck,
the Institute of Mathematics (at Kiev), 
B.~Verkin Institute for Low Temperature Physics and Engineering, and
Pidstryhach Institute for Applied Problems of Mechanics and Mathematics.

\textbf{Notation}. 
By $ \lceil x \rceil $ and $\lfloor x \rfloor$ the ceiling and floor functions are denoted, i.e., the smallest integer not less than $x$ and the greatest integer not greater than $x$, respectively.

The following sets of real and complex numbers are used: 
%the closed interval $\I = [0,\ell]$,
open half-lines $\RR_\pm = \{ x \in \RR: \pm x >0 \}$, 
%open half-planes $\CC_\pm = \{ z \in \CC : \pm \im z >0 \}$,
and open discs $\DD_\epsilon (\zeta) := \{z \in \CC : | z - \zeta | < \epsilon \}$ with the center at $\zeta$
and radius $\epsilon$.
%, the infinite sector
%\begin{equation} \label{e Sec def}
%\Sec [\xi_1,\xi_2] := \{ z \in \CC\setminus{0} \, : \, \arg z = \xi \ (\modn 2\pi)
% \text{ for certain } \xi \in [\xi_1, \xi_2] \} ,  \ \ \ \xi_1 \le \xi_2 , \ \ \ \xi_{1,2} \in \RR 
%\end{equation}
%(note that the origin $z=0$ is excluded).
The Lebesgue measure of a set $S \subset \RR$ is denoted $\meas S$.

$L_{\CC(\RR)}^p (a_1,a_2) $ are the Lebesgue spaces of complex- (resp., real-) valued functions and
\[
W^{k,p} [a_1,a_2] := \{ y \in L_\CC^p (a_1,a_2) \ : \ \pa_x^j y \in L_\CC^p (a_1,a_2), \ \ 1 \leq j \leq k \}
\]
are Sobolev spaces. The corresponding standard norms are denoted by $\| \cdot \|_p$ and $\| \cdot \|_{W^{k,p}}$. The space of continuous
complex valued functions with the uniform norm is denoted by $C
[a_1,a_2]$.

Under the support $\supp B$ of $B \in L^1 (a_1,a_2)$ we understood the topological
support of the absolutely continuous measure $|B(x)|\dd x$ on $[a_1,a_2]$, i.e.,
$\supp B$ is the smallest closed set $S$ such that $B (x) = 0 $ a.e. on $[a_1,a_2] \setminus S$.

For basic definitions of convex analysis we refer to \cite{R70}.
Let $S$ be a subset of a linear space $U$ over $\CC$ (including the case $U=\CC$).
For $u_0 \in U$ and $z \in \CC$, 
$
z S  + u_0 := \{ zu + u_0 \, : \, u \in S \}
$. 
The convex cone generated by $S$ (the
set of all nonnegative linear combinations of elements of $S$) is denoted by $\cone S$.
%The convex hull of $S$ and is denoted by $\conv S$.
Open balls in a normed space $U$ are denoted by
\[
\BB_\epsilon (u_0) =
\BB_\epsilon (u_0; U) := \{u \in U \, : \, \| u - u_0 \|_U < \epsilon
\}.
\]
The closure of a set $S$ (in the norm topology) is denoted by $\overline{S}$, the boundary of $S$ by $\Bd S$.
% in particular, $\overline {\RR}_- = (-\infty,0] $, $\Rnneg = [0,+\infty)$.
For a function $f$ defined on $S$, $f[S]$ is the image of $S$.
By $\pa_x f$, $\pa_z f $, etc., we denote (ordinary or partial)
derivatives  w.r.t. $x$, $z$, etc. The function $f_+$ is `the positive part' of f, i.e., \linebreak
%\begin{equation} \label{e f+}
$ f_+ (x):=  \left\{ \begin{array}{ll}
f(x) &,  \text{ if } f (x) \ge 0 \\
0 &,  \text{ if } f (x) < 0
\end{array}
\right.
\ ; \quad$
%\end{equation}
$\chi_S $ is an indicator function of a set $S $, i.e., $\chi_S
(x) = 1$ when $x \in S$, and $\chi_S (x) = 0$ when $x \not
\in S$. 
%The Lebesgue measure of $S \subset \RR$ is denoted by $\meas S$.

%We write $z_1^{[n]} \asymp z_2^{[n]}$ as $n \to \infty$ if the sequences
%$z_1^{[n]} / z_2^{[n]} $ and $z_2^{[n]} / z_1^{[n]}$ are bounded for $n$ 
%large enough.
%Sometimes, the complex plane $\CC$ is considered as a real linear space $\RR^2$ with scalar product
%$\lba \cdot , \cdot \rba_\CC$,
%\begin{equation} \label{e <>C}
%\lba z_1 , z_2 \rba_\CC := \re z_1 \re z_2 + \im z_1 \im z_2  , \ \ \ z_{1,2} \in \CC .
%\end{equation}
%The signum function is defined as $\sgn x := \chi_{\RR_+} (x) - \chi_{\RR_-} (x)$.

%We use the convention that the sum equals zero if the lower index exceeds
%the upper index.

\section{Basic settings, quasi-eigenvalues and admissible sets}
\label{s basics}

The boundary conditions (\ref{e bc a-}) and (\ref{e bc a+}) at $a_{1,2} \in \RR$ involve the (extended) constants $\nu_{1,2}$, which throughout the paper are assumed to satisfy 
\begin{gather} \label{e nu as}
%\text{where} 
\nu_1 \in [0,+\infty), \qquad \nu_2 \in (0,+\infty], \qquad |\nu_1| + |1/\nu_2| \neq 0, \quad
\text{ and } \nu_1 \le \nu_2  .
\end{gather}
When damping coefficients $\nu_{1,2}$ are fixed, we can define the \emph{set of quasi-eigenvalues} 
$\Si (B)$ associated with the structure $B$ as the set of $\om \in \CC \setminus \{ 0 \}$ such that  
eigenproblem (\ref{e ep}), (\ref{e bc a-}), (\ref{e bc a+}) has a nontrivial solution $y \in W^{2,1} [a_1,a_2]$ (i.e., a solution that is not identically zero).
This solution $y $ is called a (quasi-normal) \emph{mode}.
Several other names for $\om$ are used, sometimes in slightly different settings: scattering poles \cite{LP71}, dissipation frequencies \cite{KN79,KN89}, resonances \cite{F97,KS08}, quasi-normal levels (in the Quantum Physics literature). 

The function $B\in L_\RR^1 (a_1,a_2)$ in (\ref{e ep}) describes the \emph{structure} of a non-homogeneous medium varying in $x$-direction. In the case of optical cavity, $B(x_0)$ corresponds to permittivity of the transparent dielectric material of the layer with the $x$-coordinate equal to $x_0$ (if the light speed in vacuum is normalized to be $1$ or if one works with the complex wave number instead of the complex angular frequency $\om$). In Mechanics models involving the equation of nonhomogeneous string, the function $B$ is the varying linear density of the string (if the tension equals 1). So, in these models,  $B$ is supposed to be a.e. nonnegative, and we will keep this assumption in the context of optimization problems. However, in Section \ref{s mult and analyt}, the definition of quasi-eigenvalues will be extended to complex-valued coefficient $B$ for the needs of perturbation theory of Section \ref{s Pert th and weak cont}.

When $\nu_j$ is in $(0,+\infty)$, the corresponding boundary condition is dissipative and 
describes either linear damping at $a_j $ or the radiation of waves into 
outer regions (see the examples in Section \ref{ss ex A} for details). Our settings for a 
1-D resonator is a slightly generalized version of \cite{MHvG08,SSH09}. 
We allow $B$ to be equal to $0$ on certain sets (the massless string approximation). 
We also include the cases when one of the boundary conditions is conservative 
(and so the resonator is only one-side open). In particular, when $\nu_2 = +\infty$, 
we suppose that $1/\nu_2 = 0$, and then $(\ref{e bc a+})$ turns into 
the Dirichlet condition $y (a_2) = 0$. Note that the assumption $\nu_1 \le \nu_2$ does not restrict the generality since, in the case $\nu_1 > \nu_2$, one can use the change of variable $\wt x = -x$.

While most of mathematical studies were devoted to one-side open or symmetric resonators (mathematically, these two cases are almost equivalent, see Example \ref{ex sym}), 
contemporary Optics applications often involve the two-side open case 
\cite{NKT08,SSH09,KTTKRN10,BPSchZsch11,MPBKLR13}. In the present paper, we consider one- and two-side open resonators in a unified way.  Besides the importance for engineering applications, the study of two-side open resonators brings new mathematical effects. 
In particular, the spatial phase $\arg y (x)$ of the complex-valued mode $y$ is not monotone in contrast to one-side open case. 
This effect, and the fact that the notion of switch points (see \cite{Ka13}) looses its natural sense when the constraint functions $b_{1,2}$ are allowed to be equal on a set of positive measure, lead in Section \ref{s Snl pert} to more detailed study of rotational properties of resonant modes and to more essential use of Convex Analysis in comparison with the technique of \cite{Ka13}. 

 We consider optimization over the following family  of \emph{admissible structures}
\begin{equation} \label{e A}
 \A :=   \{ B(x) \in L_\RR^1 (a_1,a_2) \ : \  b_1 (x) \le B(x) \le b_2 (x) \text{ a.e. on } (a_1,a_2) \},
\end{equation}
where $ b_{1,2} $ are certain Lebesgue integrable functions defined on $(a_1,a_2)$ such that
\begin{gather}
0 \le b_1 (x) \le b_2 (x) \text{ on } (a_1,a_2) \qquad \text{ and }  
\label{e b12 cond}
 \\
\meas E > 0 \text{,  where } E = \{ x \in (a_1,a_2) \ : \ b_1 (x) < b_2 (x)\}  .
\label{e setE}
\end{gather}

A complex number $\om$ is called \emph{an admissible quasi-eigenvalue} if it belongs to the set
$
\Si [\A] := \bigcup_{B \in \A} \Si (B) .
$
Keeping in mind applications, we will pay main attention to the problem of minimization of decay rate of an individual resonance $\om$ generated by a certain $B \in \A$.
However, from mathematical point of view it is more convenient to consider $\om$ itself as a cost function with the values in $\RR^2$ and to study extremal in various senses quasi-eigenvalues . One of such sets is the boundary $\Bd \Si [\A]$ of the \emph{set of admissible quasi-eigenvalues} $\Si [\A]$. 

Actually, $\om$ can not be considered directly as a functional of $B$ even locally due to 
multiplicity and splitting issues (see Sections \ref{ss mult}, Proposition \ref{p per k mult}, 
and the discussion in \cite{Ka14}). However, a rigorous approach to $\om (B)$ can 
be given via the set-valued map $B \mapsto \Si (B)$. Local extrema of such set-valued maps 
were introduced in \cite{Ka14} and will be the key tool in Section \ref{s Opt=>nl}. 
In such generalized settings, $\Bd \Si [\A]$ plays a role of a  set of 
generalized Pareto extremizers for the map $B \mapsto \Si (B)$ over $\A$.

\section{Properties of resonances and related maps}
\label{s mult and analyt}

\subsection{Characteristic determinant $F$ and definition of multiplicity}
\label{ss F}

In this section, an equivalent (but more convenient from the point of view of perturbation theory) definition of quasi-eigenvalues is given under more general assumption that $B \in L^1_\CC (a_1,a_2)$.

Denote by $\vphi (x ) = \vphi (x, z; B )$, $\psi(x) = \psi (x, z; B)$, and $\theta (x) =\theta(x, z; B )$ the solutions to
$\pa_x^2 y(x) = - z^2 B(x) y(x) $ satisfying
\begin{eqnarray} \label{e phi psi}
 \vphi (a_1, z; B )  = \pa_x \psi (a_1, z; B) = 1 , & \ \ & \pa_x
\vphi (a_1, z; B ) = \psi (a_1, z; B ) = 0 ,  \notag \\
\theta (a_1,z;B) = 1, & \ \ & \pa_x \theta (a_1,z;B) = -\ii z \nu_1 . \label{e th}
\end{eqnarray}

Obviously,
$\theta (x) =  \vphi (x) - \ii z \nu_1 \psi (x)$ and $\theta (x)$ is a unique solution to the
integral equation
\begin{eqnarray}
 y (x) & = & 1 - \ii z \nu_1 (x-a_1) - z^2 \ \int_{a_1}^x (x-s) \ y (s) \ B (s) \dd s ,
\quad a_1 \leq x \leq a_2. \label{e int ep th}
\end{eqnarray}

Recall that the set of quasi-eigenvalues corresponding to a structure $B$ (in short, quasi-eigenvalues of $B$)
is denoted by $\Sigma(B)$.
The following lemma is the integral reformulation of the quasi-eigenvalue problem.

\begin{lem} \label{l int}
A number $\om \in \CC$ belongs to $\Si (B)$ if and only if
there exists nontrivial $y (x) \in C [0,\ell]$
satisfying equation (\ref{e int ep th}) with $z = \om$ and the equality
\begin{equation} \label{VGUI}
y(a_2) + \frac{\nu_1}{\nu_2}-\frac{\ii \om}{\nu_2}\int_{a_1}^{a_2} y (s) \ B (s) \dd s =0.
\end{equation}
% $y (x) \in W^{2,1} [a_1,a_2]$.
\end{lem}

Note that in the integral settings there is no need to exclude separately the case $\om = 0$.
If $\om =0$, then it is easy to see that the problem consisting of (\ref{e int ep th})
with  $z = \om$  and (\ref{VGUI}) has no nontrivial solutions.

Clearly, $\Sigma (B)$ is the set of zeroes of the function $F (\cdot;B)$, where
\begin{equation} \label{e F}
F (z) = F (z; B) := \theta(a_2,z;B) + 
\frac{\nu_1}{\nu_2}-\frac{iz}{\nu_2}\int_{a_1}^{a_2} \theta(s,z;B) \ B (s) \dd s .
\end{equation}
It is easy to see that for $z \neq 0$,
\begin{equation} \label{e F th}
F (z; B) = \theta (a_2,z;B) + \frac{\ii \pa_x \theta (a_2,z;B)}{z\nu_2} .
\end{equation}
When $\nu_2 = \infty$, these formulas turns into $F (z; B) = \theta (a_2,z;B) $.

We say that a map $G : U_1 \to U_2$ between normed spaces $U_{1,2}$
is \emph{bounded-to-bounded} if the set $G[S]$ in $U_2$ is bounded for
any bounded $S$ in $U_1$.

\begin{lem} \label{l an}
$(i)$ The functional $F (z; B) $
%:= \vphi (\ell,z; B)  +  \ii z \int_{0-}^{\ell+}  \vphi (s, z; B) \ B (s)$
is analytic on the Banach space $\CC \times L^1 (a_1,a_2)$.

$(ii)$ Let $1\le p \le \infty$. The mappings $(z, B) \mapsto \vphi (\cdot ,z; B)$, $(z, B) \mapsto \psi (\cdot ,z; B)$,
and $(z, B) \mapsto \theta (\cdot ,z; B)$ are
bounded-to-bounded and analytic from $\CC \times L_\CC^p (a_1,a_2)$ to
$W^{2,p} [a_1,a_2]$.
\end{lem}

The proof is given in Appendix \ref{ss an and exist} (concerning analytic mappings on Banach spaces, see 
\cite{PT87}).

It is obvious that all modes $y$ corresponding to $\om \in \Si (B)$
are equal to $\theta (\cdot, \om; B )$ up to a multiplication by a
constant.  So \emph{the geometric multiplicity} of any
quasi-eigenvalue equals 1. In the following, \emph{the multiplicity}
of a quasi-eigenvalue means its \emph{algebraic multiplicity}.

 \begin{defn}\label{d mult}
\emph{The multiplicity} of a quasi-eigenvalue of $B$ is its multiplicity as
a zero of the entire function $F (\cdot; B) $. A quasi-eigenvalue is
called \emph{simple} if its multiplicity is $1$.
The set of non-simple  quasi-eigenvalues is denoted by $\Si_\mult (B)$ 
(non-simple quasi-eigenvalues are often called multiple, or degenerate).
 \end{defn}

This is essentially the classical M.V. Keldysh definition of multiplicity for eigenvalue problems 
with an eigen-parameter in boundary conditions 
(note that $F$ coincides with the characteristic determinant of (\ref{e ep})-(\ref{e bc a+}) up to a constant), see 
\cite{KN79,KN89,CZ95,PvdM01,Ka13,Ka14} and references therein.

Since $F (0; B) =1 + \frac{\nu_1}{\nu_2} >0$, each quasi-eigenvalue has a finite
multiplicity and the set $\Si (B)$ consists of isolated points, which can accumulate only to $\infty$. Note that $\Sigma (B)$ may be empty, see Proposition \ref{p const struct}.

\subsection{Non-simple resonances, the case of nonnegative $B$, and examples}
\label{ss mult}

\emph{There exist triples} $(\nu_1,\nu_2,B)$, consisting of numbers $\nu_{1,2}$
satisfying the assumption (\ref{e nu as}) and a nonnegative function $B$,
\emph{that generate non-simple quasi-eigenvalues}.
This follows from \cite{GP97}
 and \cite[Theorem 4.1]{PvdM01} (with $\nu_2 =+\infty$).
For a slightly different class of quasi-eigenvalue problems, 
existence of degenerate quasi-eigenvalues was proved in \cite[Theorem 3.1]{KN89} 
(English translation of this theorem can be found in \cite{Ka13_KN}, see also \cite{KN79}), and 
examples were given in \cite{MvdBY01,Ka13_KN}.

\begin{lem} \label{l prop Si}
Let $B(x) \ge 0$ a.e. and $\nu_{1,2}$ satisfy (\ref{e nu as}). Then:
\item[(i)] $\Si (B) \subset \CC_-$,
\item[(ii)] $\Sigma(B)$ is symmetric w.r.t. the imaginary axis $\ii \RR$, moreover,
the multiplicities of symmetric quasi-eigenvalues coincide.
\end{lem}

\begin{proof}
\textbf{(i)}
Let $\om \in \Si (B) $ and $y$ be an associated eigenfunction.
The energy of the eigen-oscillation $u(x,t)= e^{-\ii \om t} y (x)$ at the time $t$ is
\[
\int_{a_1}^{a_2} |\pa_x u|^2 + B |\pa_t u|^2 \dd x = e^{- 2 \beta t} \En_y , \text{ where }
\En_y := \int_{a_1}^{a_2} |y'|^2  + B  \ |\om|^2 \ |y|^2 \dd x .
\]
Since $y$ is nontrivial, we see that $\En_y > 0$.
(Otherwise, $y' \equiv 0$ and, in turn, $y \equiv 0$).

Recalling that $\om \neq 0$ and taking the real part in the identity
\begin{gather*}
-\ii \om \En_y = -\ii\om \int_{a_1}^{a_2} |y'|^2 \dd x +\ii\overline{\om}\int_{a_1}^{a_2} (-\om^2) B \ y \ \overline{y} \ \dd x
= \\ =
- 2 \ii \re \om \int_{a_1}^{a_2} |y'|^2 \dd x - |\om|^2 \left( \nu_1 |y(a_1)|^2 +\frac{ |y'(a_2)|^2}{|\om|^2\nu_2} \right) ,
\end{gather*}
one gets
\[
\im \om  = - \frac{|\om|^2}{\En_y} \left(\nu_1 |y(a_1)|^2 +  \frac{|y'(a_2)|^2}{|\om|^2\nu_2} \right) < 0 .
\]
Note that inequality
$ \nu_1 |y(a_1)|^2 +\frac{ |y'(a_2)|^2}{|\om|^2\nu_2}  >0$ follows from (\ref{e nu as}) and $y \not \equiv 0$.

\textbf{(ii)} follows from identities
\begin{equation} \label{e iR real}
\text{$\theta (x, \ii \zeta;B)=\overline {\theta (x, \ii \overline{\zeta};B)}$
and $F (\ii \zeta ; B) = \overline{F (\ii \overline{\zeta})}$.}
\end{equation}
\end{proof}

When the medium of a resonator is homogeneous, quasi-eigenvalues can be calculated explicitly.

\begin{prop} \label{p const struct}
Let $B (x) \equiv b$, where $b \ge 0$ is a constant.
\item[(i)] If $b^{1/2}=\nu_1$ or $b^{1/2}=\nu_2$, then $\Si (b) = \varnothing$.
\item[(ii)] If $b >0$, $b^{1/2} \neq \nu_1$, and $b^{1/2} \neq \nu_2$, then
$\Si (b) = \{ \om_n (b) \}_{n=-\infty}^{+\infty} $ with
\begin{equation} \label{e om n hom}
 \om_n  (b) =
-  \frac{ \ii }{ 2 b^{1/2} (a_2 -a_1) }
\ln \left| \frac{1+K_1}{1-K_1} \right| + \frac{\pi}{b^{1/2} (a_2 -a_1) } \times \left\{
\begin{array}{ll}
n , & \text{if } b^{1/2} \not \in [\nu_1, \nu_2] \\
(n+1/2) , & \text{if } b^{1/2} \in (\nu_1 , \nu_2)
\end{array} \right. ,
\end{equation}
where $ K_1
%=  K_1 (b) = K (b,\nu_1,\nu_2)
:=  \frac{1+\nu_1/\nu_2 } {\frac{\nu_1}{b^{1/2}} + \frac{b^{1/2}}{\nu_2}}  >0$ and $n \in \ZZ$.

\item[(iii)] If $b=0 \neq \nu_1$, then $\Si (b) = \{ - \ii \frac{1/\nu_1 + 1/\nu_2} {a_2-a_1} \}$.

In the cases (ii)-(iii) all quasi-eigenvalues are simple.
\end{prop}

\begin{proof}
Consider the case $b>0$. The equation $F(z) = 0$ can be transformed into
$\tan ([a_2-a_1]b^{1/2}z) + \ii K_1 = 0$ and we know that the roots $\om$ are in $\CC_-$. Taking $\pa_z$ one can see that the roots are simple.
We have $K_1 = 1$ exactly  when $b^{1/2} \in \{ \nu_1,\nu_2\}$. In this case $\Si (b) = \varnothing$. In the case $K_1 \neq 1$, the real part of $\Arctan (-\ii K_1)$ is either
$\pi n$ or $\pi (n+1/2)$ depending on the sign of $1-K_1$. This leads to
(\ref{e om n hom}). Note that
\begin{equation} \label{e sign of 1-K1}
1>K_1 \ \Leftrightarrow \ b^{1/2} \not \in [\nu_1, \nu_2], \quad 1<K_1 \ \Leftrightarrow \ b^{1/2} \in (\nu_1, \nu_2),  \quad \text{ and } K_1 =1 \ \Leftrightarrow \ b^{1/2}  \in \{\nu_1, \nu_2\} .
\end{equation}

The case $b=0$ can be treated by straightforward computations.
\end{proof}
In the particular cases of one-side open resonators (i.e., when $\nu_1 =0$ or $\nu_2=+\infty$) this example is well known \cite{U75,CZ95}.

\section{Perturbations and weak continuity of resonances}
\label{s Pert th and weak cont}

Let $U$ be a Banach space. For a functional $G (z;u)$ that maps
$\CC \times U$ to $\CC$,
we denote by
\[
 \frac{\pa G (z;u_0)}{\pa u} (u_1) :=
 \lim_{\zeta \to 0} \frac{G (z;u_0 + \zeta u_1 ) - G (z;u_0)}{\zeta} ,
 \]
\emph{the directional derivative} of $G(z;\cdot)$ along the vector $u_1 \in U$ at the point $u_0 \in U$.
By $ \frac{\pa G (z;u_0)}{\pa u} $ the corresponding \emph{Fr\'echet derivative in $u$} is denoted.
If the Fr\'echet derivative $\frac{\pa G (z;u_0)}{\pa u} $ exists, it is a linear functional on $U$ and $ \frac{\pa G (z;u_0)}{\pa u} [S]$ defines the image of a set $S \subset U$
under this functional.

\begin{rem}[Global discontinuity of $B \mapsto \Si (B)$] 
\label{r glob discont}
Before the study of continuity properties of quasi-eigenvalues,  let us note that \emph{new quasi-eigenvalues can appear from $\infty$ with small variations of $B$}. To see this, it is enough to take the homogeneous structure $B_0 $ equal to $\nu_1^2$ or $\nu_2^2$ and to consider $\Si (B)$  using (\ref{e om n hom}) and taking $B\equiv B_0 + \ep$ with small $\ep \in \RR$. 

This case is very special because, when $B_0 = \nu_j^2$, the resonator is not separated from outer space by any reflecting barrier (see detailed explanations in \cite{CZ95}). We \emph{do not know if the above effect of global in $\om$ discontinuity of the set-valued map $B \mapsto \Si (B)$ can happen in a vicinity of $B_0 \not \equiv \nu_{1,2}^2$}.
\end{rem}

\subsection{Local weak continuity}
\label{ss weak cont}

Let us fix a countable family $\{ f_n \}_{n=1}^\infty$ of continuous functions that is dense in $C[a_1,a_2]$.
This family generates a metric
$\displaystyle \rho_\Mes (\dd M_1 , \dd M_2) := \sum_{n=1}^\infty \frac{|\int f_n (\dd M_1 - \dd M_2) |}{2^n (1+|\int f_n (\dd M_1 - \dd M_2) |)}$
on the space $\Mes$ of complex Borel measures  on $[a_1,a_2]$. The weak* topology on any closed ball in $\Mes$ coincides with
the topology generated by the metric $\rho_\Mes $. We will use the metric
\[
\rho (B_1, B_2) := \rho_\Mes (B_1 \dd x , B_2 \dd x ) \qquad \text{ on } L^1_\CC (a_1,a_2) ,
\]
where $B_j \, \dd x$ are absolutely continuous measures corresponding to the functions $B_j \in L^1_\CC$.

In this subsection, $\BB_R $ is the closed ball $\BB_R (0)$ in $L^1_\CC (a_1,a_2)$ with radius $R>0$ and centered in the origin.

\begin{lem} \label{l weak cont}  Let $z \in \CC$.
\item[(i)] The functional $F (z; \cdot) : \BB_R \to \CC$ is continuous w.r.t. the metric $\rho$.
\item[(ii)] The map $B \mapsto \theta (\cdot,z;B) $ is continuous from
the metric space $(\BB_R,\rho)$ to the normed space $C [a_1,a_2]$.
\end{lem}

\begin{proof}
\textbf{(i)} follows from (ii). Let us prove \textbf{(ii)}.
Suppose $B_*, B_j \in \BB_R$ and $\rho (B_j ,B_*) \to 0$ as $j \to \infty$. In particular,
$\wlim B_j \dd x = B_* \dd x$. By Lemma \ref{l an} (ii),
\begin{equation} \label{e th j comp}
\text{the sequence $\theta_j (\cdot) := \theta (\cdot,z;B_j)$ is bounded in $W_\CC^{2,1}$, and so is relatively compact in $C [a_1,a_2]$.}
\end{equation}
Hence there exists a subsequence $\theta_{j_k}$ (strongly) convergent in $C [a_1,a_2]$ to a certain $\theta_*$.
Using $\wlim B_j \dd x = B_* \dd x$ and passing to the limit in (\ref{e int ep th}) for every $x \in [a_1,a_2]$,
one can show that $\theta_* (\cdot) = \theta (\cdot,z;B_*)$.

Assume that $\theta_j$ does not converge to $\theta_*$. Then (\ref{e th j comp}) imply that
there exists a subsequence $\theta_{n_k} $ convergent to $\theta_{**} \neq \theta_*$.
Passing to the limit in (\ref{e int ep th}) again, we see that $\theta_{**} (\cdot ) = \theta (\cdot,z;B_*) = \theta_* (\cdot) $,
a contradiction. Thus, $\theta (\cdot,z;B_j) \to \theta (\cdot,z;B_*)$ as $j \to \infty$.
\end{proof}

The \emph{total multiplicity of quasi-eigenvalues of $B$ in a set} $\D \subset \CC$ is the sum of multiplicities
of all $\om \in \Si (B) \cap \D$.

\begin{prop} \label{p weak cont mult}
Let $B_0 \in L^1_\CC (a_1,a_2)$ and $R \ge \| B_0 \|_1$. Let $\D$ be an open bounded subset of $\CC$ such that its boundary
$\Bd \D $ does not contain quasi-eigenvalues of $B_0$. Then
there exists a neighborhood $W \subset \BB_R$ of $B_0$ in the topology of the metric space $(\BB_R,\rho)$
such that, for any $B \in W$,
\begin{eqnarray*}
& \text{there are no quasi-eigenvalues of $B$ on $\Bd \D $,} & \\
& \text{the total multiplicity of quasi-eigenvalues of $B$ in $\D$ coincides with that of $B_0$.} &
\end{eqnarray*}
\end{prop}

\begin{proof}
By Lemmas \ref{l an} and \ref{l weak cont}, the functional $F(z;B)$ is analytic in $z$ and $B$ and is continuous
in the second variable on the metric space $(\BB_R,\rho)$.
This allows us to apply \cite[Theorem 9.17.4]{D69} (which is a corollary from Rouch\'e's theorem) to $F(z;B)$.
The equivalence $z \in \Si (B) \Leftrightarrow F (z;B) = 0$ completes the proof.
\end{proof}

\subsection{Resonances' perturbations and directional derivatives of $F$}
\label{ss Pert}

\begin{lem} \label{l der F}
At quasi-eigenvalues $\om \in \Sigma (B)$, the derivative of $ F $ w.r.t. $z$ is given by
\begin{equation} \label{e paz F}
\frac{\pa F (\om;B)}{\pa z} = \frac{2\om}{\partial_x\theta(a_2,\om,B)}   \int_{a_1}^{a_2} \theta^2 (s,\om;B) B (s) \dd s+\frac{i\nu_1}{\partial_x\theta(a_2,\om,B)}+\frac{\theta(a_2,\om,B)}{\om}
;
\end{equation}
and the directional derivatives $[\pa_{B} F (\om,B)] \ (V)$ w.r.t. $B$ in the direction $V \in L_\CC^\infty (a_1,a_2)$ by
\begin{equation} \label{e paBD F}
\frac{\pa F (\om,B)}{\pa B} (V) = \frac{\om^2}{\partial_x\theta(a_2,\om,B)} \int_{a_1}^{a_2}
\theta^2 (s,\om;B) \ V (s) \ \dd s .
\end{equation}
\end{lem}

The calculation of derivatives resembles that of \cite{Ka13} and is given in Appendix \ref{ss proof of der F}.

\begin{prop} \label{p per k mult}
Let $\om $ be a quasi-eigenvalue of $B \in L^1_\CC$ of multiplicity $m \in \NN$.
Then:
\item[(i)] For every direction $V \in L^1_\CC$ , there exist $\ep>0$, $\de>0$,
and  functions $\Om_j $, $j=1,\dots,m$, continuous on $[0,\de)$
such that for $ \zeta \in [0,\de) $, all the quasi-eigenvalues of $B + \zeta V $
lying  in $\DD_{\ep} (\om)$ are given by $\{ \Om_j (\zeta) \}_1^m$
taking multiplicities into account.

\item[(ii)] For each of these functions the following asymptotic formula is valid
\begin{equation} \label{e Om asy}
\Om_j (\zeta) = \om + [K (\om,B;V) \zeta]^{1/m} + o (\zeta^{1/m})
\qquad   \text{ as $\zeta \to 0$,  \ \ with }
K (\om,B;V) := -\frac{m! \, \frac{\pa F (\om,B)}{\pa B} (V)}{\pa_z^m F (\om ; B)}  .
\end{equation}

\item[(iii)]
In the case $K (\om,B;V) \neq 0$, each branch of $ [\cdot]^{1/m}$ corresponds
to exactly one of functions $\Om_j$. So for small enough $\zeta>0$,
all $m$ values of functions  $\Om_j (\zeta)$ are distinct and simple quasi-eigenvalues of $B + \zeta V$.
\end{prop}

\begin{proof}
For a one-side open resonator, an analogue of the particular case considered in (iii) 
was obtained in \cite[Proposition 3.5]{Ka13}.
Statement (i) is well known for resonances of the Schr\"odinger
equation and can be obtained
using the Weierstrass preparation theorem and the Puiseux series theory.
We start from  these arguments to prove that (\ref{e Om asy}) holds always (not only  when $K (\om,B,V) \neq 0$ as in \cite{Ka13}).

Consider the entire function $Q (z,\zeta) := F ( z ; B + \zeta
V )$ of two complex variables $z$ and $\zeta$. Then $\om $ is
an $m$-fold zero of the function $Q (\cdot,0)$. By the
Weierstrass preparation theorem, in a certain polydisc $\D_{\ep}
(\om ) \times \D_{\delta_1} (0)$,
\[
Q (z,\zeta) = %P (z-\om,\zeta)
[(z-\om )^m + h_1 (\zeta) (z-\om)^{m-1} + \dots + h_m (\zeta) ] R (z,\zeta) ,
\]
where the coefficients $h_j$ (the function $R (z,\zeta)$) are
analytic in $\D_{\delta_1} (\om)$ (resp., in $\D_{\ep} (\om)
\times \D_{\delta_1} (0)$),
\[
\text{$h_j (0) = 0$, \ \ \ and  \ \ \ $R (z,\zeta) \neq 0
$ \ in \ $\D_{\ep} (\om) \times \D_{\delta_1} (0)$.}
\]
Differentiating $Q$ by $\zeta$, one gets
$
\frac{\pa F (\om,B)}{\pa B} (V) = \pa_\zeta Q  (\om,0)  = h'_m (0) R (\om,0)  .
$
On the other side,
\[
\pa_z^m F (\om,B) = \pa_z^m Q (\om,0) = m! R (\om,0) \neq 0.
\]
Hence,
\begin{equation} \label{e hm in 0}
h'_m (0) =
\frac{\pa_\zeta Q  (\om,0) }{ R (\om,0)} =
 \frac{m! \frac{ \pa F ( \om ; B )}{\pa B} (V)} { \pa_z^m F (\om,B)} = - K (\om,B;V).
\end{equation}

Denote by $P (z,\zeta) :=(z-\om )^m + h_1 (\zeta) (z-\om)^{m-1} + \dots + h_m (\zeta) $ the  Weierstrass polynomial.  Since $R (z,\zeta) \neq 0$, the zeros of $Q$ and $P$ in  $\D_{\ep_1} (\om) \times \D_{\delta_1} (0)$ coincide.
It is well known that for small enough $\delta \in (0,\delta_1]$
and $|\zeta | < \delta $, there are exactly $m$ $z$-roots (counting
multiplicity) of $P (z,\zeta) = 0$ and these roots are given by a one or more convergent
(and possibly multivalued) Puiseux series. We denote these $m$ branches of (one or several)  Puiseux series by $\Om_j (\zeta)$.

The method of the Newton diagram (see e.g. \cite{VT74} and \cite[Theorem XII.2]{RS78IV})  implies that, if the roots $\Om_j (\zeta)$ are given by more than one
 Puiseux series, then $h'_m (0) = K (\om,B,V) =0$ and
all the roots have the asymptotics
$ \Om_j (\zeta) = \om + o (\zeta^{1/m})$.

In the case when the roots $\Om_j (\zeta)$ are given by one  Puiseux series,
this series can be written in powers of $\zeta^{1/m}$ and formula
(\ref{e Om asy}) follows from  equalities
(\ref{e hm in 0}) and 
$ (\Om_j-\om )^m + h_1 (\zeta) (\Om_j -\om)^{m-1} + \dots + h_m (\zeta) = 0 $, 
see details in the proof of \cite[Proposition 3.5]{Ka13}.
\end{proof}

\section{Variational characterization of the nonlinear spectrum}
\label{s Snl pert}

Recall that $\ChiCpl$ is the indicator function of $\CC_+$.
A function $y \in W^{2,1}_\CC [a_1,a_2]$ is called a nontrivial solution to
\begin{equation} \label{e nonlin}
y'' (x)  = - \om^2 y(x) \left[ b_1 (x) + [b_2(x) - b_1(x)] \ChiCpl (y^2 (x) ) \right]
\end{equation}
if $y \not \equiv 0$ and (\ref{e nonlin}) is satisfied for a.a. $x \in [a_1,a_2]$.
%\begin{defn} \label{d nonlin eig}
If problem (\ref{e nonlin}), (\ref{e bc a-}), (\ref{e bc a+}) admits a nontrivial
solution $y$, we say that $y$ is an eigenfunction and $\om$ is an eigenvalue of 
(\ref{e nonlin}), (\ref{e bc a-}), (\ref{e bc a+}). In short, we will say that $\om$ \emph{is a 
nonlinear eigenvalue. The set of nonlinear eigenvalues is denoted by $\Si^\nl $}.
%\end{defn}

Note that
\begin{equation} \label{e Snl sym}
\text{$\Si^\nl$ is symmetric w.r.t. $\ii \RR$.}
\end{equation}
Indeed, if $\om=\ii \zeta \in \Si^\nl$ and $y$ is an associated eigenfunction, then for
$B = [ b_1 + (b_2 - b_1) \ChiCpl (y^2 ) ] $ one has $\om   \in \Si (B)$.
Since  $\om_1 = \ii \overline{\zeta} \in \Si (B)$ and $\overline{y}$ 
is an  eigenfunction associated with $\om_1$ (see Lemma \ref{l prop Si} (ii) an its proof), we conclude that  $\om_1 $ is a nonlinear eigenvalue
with $y_1 := \ii \overline{y}$ as an eigenfunction. 

Let us notice also that $B = [ b_1 + (b_2 - b_1) \ChiCpl (y^2 ) ]$ belongs to $\A$ and, 
moreover, belongs to the set $\Ext \A$ of extreme points of $\A$. So,
\begin{equation} \label{e Snl subset}
\Si^\nl \subset \Si [\Ext \A] \subset \Si [\A].
\end{equation}

Let $K$ be the 
functional defined in Proposition \ref{p per k mult} (ii) and let 
$\cone K (\om, B; \A - B)$ be the nonnegative convex cone  generated by the set 
$K (\om, B; \A - B) = \{ K (\om, B; \wt B-B) \ : \ \wt B \in \A\} $.

\begin{thm} \label{t char Si nl}
Let $\om \not \in \ii \RR$. Then the following statements are equivalent:
\item[(i)] $\om $ is nonlinear eigenvalue,
\item[(ii)] there exists $B_0 \in \A$ such that $\om \in \Si (B_0)$ and 
$\cone K (\om, B_0; \A - B_0) \neq \CC$.

If (i)-(ii) hold true, then there exists an eigenfunction $y$ of (\ref{e nonlin}), (\ref{e bc a-}), (\ref{e bc a+})
such that
\begin{equation} \label{e B0=b1b2}
B_0 (x) = \left[ b_1 (x) + [b_2(x) - b_1(x)] \ChiCpl (y^2 (x) ) \right] \text{ a.e. on } (a_1,a_2).
\end{equation}
\end{thm}

The proof (see Section \ref{ss proof of Si nl}) relies on the rotation properties of $\theta$
 described in Section \ref{ss turn int}.

\subsection{The turning interval $[x_*,x^*]$ and rotational properties of modes.}
\label{ss turn int}

By $\arg z$ we understood the multivalued argument function of a complex variable $z \neq 0$. 
In this subsection, we write $\theta (x)$ for $\theta (x,\om;B)$.

\begin{lem} \label{l turn int}
Let $B \in L^1$ and  $\om \in \Sigma(B)$. Assume that $\alpha = \re \om > 0$ and $B(x) \ge 0$ a.e. on $(a_1,a_2)$.
Then:
\item[(i)] There exists a  unique subinterval $[x_* , x^*]$ of $[a_1,a_2]$
with the properties that
\begin{multline}
\text{$\overline{\theta (x)} \pa_x \theta (x) \in \CC_-$
when $x < x_*$,} \qquad \text{$\overline{\theta (x)} \pa_x\theta(x)\in \RR$ when $x \in [x_*, x^*]$,} \qquad \\
\text{and \quad $\overline{\theta (x)} \pa_x\theta (x) \in \CC_+$  when $x>x^*$.}
\end{multline}
(It is supposed that $x_* \le x^*$ and that, in the case $x_*=x^*$, the \emph{turning} interval $[x_* , x^*]$
degenerates into a single \emph{turning} point $x_* = x^*$.
The meaning of 'turning' is explained by statement (\ref{e arg dec arg inc}) below).

\item[(ii)] If $x_* < x^*$, then
\begin{equation} \label{e B=0 th=const}
\text{
$B = 0$ a.e. on $(x_*, x^*)$ and $\pa_x \theta (x)$ is constant on $ [x_*, x^*]$.
}
\end{equation}
In particular, $B > 0$ a.e. yields $x_* = x^*$.
\item[(iii)] There exists at most one point $x_0 \in [a_1,a_2]$ such that $\theta (x_0) = 0$. If $x_0$ does exist,
it belongs to $[x_*, x^*]$.
(In the sequel, when $x_0$ does not exist, we assume $\{ x_0 \} = \varnothing$.)

\item[(iv)] The multifunction $\arg \theta (x)$ has a branch $\arg_* \theta (x)$ that is defined on $[a_1,a_2] \setminus \{x_0\}$
and has the following properties:
\begin{gather}
\text{ $\arg_* \theta (x)$ is differentiable on $[a_1,a_2] \setminus \{x_0\}$, } \notag \\
\text{$\pa_x \arg_* \theta < 0$ if  $x < x_*$, \quad
$\pa_x \arg_* > 0 $ if $x > x^*$, \quad 
$\pa_x \arg_* \theta = 0$ if  $x \in [x_* ,x^*] \setminus \{x_0\}$} .
\label{e arg dec arg inc}
%\\ \text{$\arg_* \theta$ is constant on each of the connected components of $$.} \notag
\end{gather}

\item[(v)] $\pa_x \arg_* \theta (x)$ is bounded on $[a_1,a_2] \setminus \{x_0\}$. In particular,
for every $\xi \in [-\pi,\pi)$,
\[
\text{the set $\{ x \in [a_1,a_2] \setminus [x_*,x^*] \ : \ \arg \theta (x) = \xi (\modn 2\pi) \}$
is at most finite.}
\]
\end{lem}

\begin{proof}
\textbf{(i)-(ii)} Since $\im \om <0 $ and $\re \om >0$, we have $\im \om ^2<0$ and
\begin{equation} \label{e im th path}
\pa_x \im (\overline{\theta}\pa_x\theta ) =
\im (\overline{\theta}\pa_x^2\theta)=\im (-\overline{\theta}\om^2 B \theta)=
B |\theta|^2 \im (-\om^2) \ge 0 \text{ a.e. }
\end{equation}
For $x=a_1$ and $x=a_2$ using (\ref{e th}) and (\ref{e bc a+}), one gets
$$\im [\overline{\theta(a_1)} \pa_x\theta(a_1)] = \nu_1\im(-\ii \om )=-\nu_1\alpha \le 0 $$
and
$$\im [\overline{\theta(a_2)} \pa_x\theta(a_2) ] =\im
\left[ \frac{\pa_x \overline{\theta(a_2)}}{-\ii \overline{\om} \nu_2} \pa_x \theta(a_2) \right]
=\frac{|\pa_x \theta(a_2)|^2}{\nu_2} \im \frac 1{ -\ii \overline{\om}} \ge 0.
$$
Hence, $\im( \overline{\theta}\pa_x\theta)$ is a nondecreasing continuous function having a zero in $[a_1, a_2]$.
Denote
\[
\text{$x_* = \min \{x : \im [\overline{\theta(x)} \pa_x\theta(x) ] = 0\}$
and $x^* = \max \{x : \im [\overline{\theta(x)} \pa_x\theta(x) ] = 0\}$.
}
\]
Then
\begin{equation} \label{e >0 <0 out x*x*}
\text{ $\im[\overline{\theta(x)}\pa_x \theta(x)]  < 0 $ for $x < x_* $, \quad and \quad
$\im[\overline{\theta(x)}\pa_x \theta(x)]  > 0 $ for $x > x^* $. }
\end{equation}

When $x_* < x^*$, (\ref{e im th path})  implies
\begin{equation} \label{e im=0}
\text{
$\im[\overline{\theta(x)}\pa_x \theta(x)] = 0 $ for all $x \in [x_*, x^*] $.
}
\end{equation}
Combining this with $\pa_x^2 \theta = - \om^2 B \theta$, we get
(\ref{e B=0 th=const}).

\textbf{(iii)} It follows from (i) that every $x$-zero of $\theta $ belongs to $[x_*,x^*]$.
If there exist two distinct zeroes, then (\ref{e B=0 th=const}) yields $\pa_x \theta \equiv 0$ on $[x_*,x^*]$ and so
$\theta \equiv 0$ on $[a_1,a_2]$, a contradiction (to (\ref{e th})).

\textbf{(iv)}
Consider an arbitrary continuous on $[a_1,a_2] \setminus \{ x_0 \}$ branch $\arg_* \theta (x)$
of $\arg \theta (x)$.
Combining the equality
\begin{equation} \label{e pax arg}
\frac{\im (\overline{\theta} \pa_x \theta) }{|\theta|^2} =
\im \frac{\pa_x \theta }{\theta} = \im \pa_x \ln \theta  = \pa_x \arg_* \theta
 \quad \text{ for } x \neq x_0,
\end{equation}
with
(\ref{e >0 <0 out x*x*}) and (\ref{e im=0}) we get (iv).

\textbf{(v)} %It follows from (\ref{e pax arg}) that $\pa_x \arg_* \theta$ is bounded on any set with positive distance
%to $x_0$ (and is bounded on $[a_1,a_2]$ if $x_0$ does not exist).
Due to formula (\ref{e pax arg}), it is enough to prove that $\pa_x \arg_* \theta$
is bounded in a punctured neighborhood of $x_0$ assuming that $x_0$ exists.

To be specific, consider the case $x>x_0$.
Since $\theta (x_0) = 0$, one has
$\theta' (x_0) \neq 0$ and
$\theta (x) = \theta' (x_0) (x-x_0) - \om^2 \int_{x_0}^x (x-s) \theta (s) B (s) \dd s$.
Hence,
\begin{multline} \label{e pax arg int}
\pa_x \arg_* \theta = \frac{\im (\overline{\theta} \pa_x \theta) }{|\theta|^2} =
|\om|^4 \int_{x_0}^x \int_{x_0}^x B (s) B (t) \im \frac{(x-s) \overline{\theta (s)} \theta (t)}{|\theta (x)|^2} \ \dd s \ \dd t
= \\
= |\om|^4 \int_{x_0}^x \int_{x_0}^x B (s) B (t) \im \frac{(x-s) \overline{\theta (s)} (t-s)  \theta' (p(t,s))}{|\theta (x)|^2} \ \dd s \ \dd t ,
\end{multline}
where $p(t,s)$ lies between $t$ and $s$.
Since $\frac{(x-s) (t-s)}{|\theta (x)|^2} $ is bounded in the square $(s,t) \in (x_0,x)^2 $ (recall that $\theta' (x_0) \neq 0$)
and $\theta, \theta' \in C [a_1,a_2]$,
we see that the last integral in (\ref{e pax arg int}) is uniformly bounded for $x>x_0$.
\end{proof}

Now consider modes $\theta$ corresponding quasi-eigenvalues on the axis $\ii \RR$.
In this case, the analysis is slightly different, but simpler, since
\begin{equation} \label{e theta real}
\theta(x,z;B) \text{ are real when } \ \ z \in \ii \RR
\text{ and } B \in L^1_\RR (a_1,a_2).
\end{equation}

\begin{lem} \label{arg th with alf=0}
Let $\om \in \Sigma(B)$, $ \om = \ii\beta,\ \beta\in\RR$, and   $B(x)\geq0 $ a.e.  on $(a_1,a_2)$. Then:
\item[(i)] There exists a  unique subinterval $[x_*,x^*]$ of $[a_1,a_2]$ such that
%$$\theta (x) \pa_x\theta(x)<0\ \  for \ x<x_*, \ \theta (x) \pa_x\theta(x)>0 \ \ for \ x>x_{**}$$
%and  $\theta (x) \pa_x\theta(x)=0 \ for \ x\in[x_*,x_{**}].$
\begin{multline*} \label{e th pax th iR}
\text{$\theta (x) \pa_x\theta(x)<0$
if $x < x_*$,} \qquad \text{$\theta (x) \pa_x\theta(x)=0$ if $x \in [x_*, x^*]$,} \qquad \\
\text{and \quad $\theta (x) \pa_x\theta(x)>0$  if $x>x^*$.}
\end{multline*}

\item[(ii)] If $x_*<x^*$,  then $B(x)=0$  a.e. on $(x_*,x^*)$ and
$\theta (x)$ is  a nonzero constant function on $[x_*,x^*].$
In particular, if  $B(x)>0$  a.e. on $(a_1,a_2)$, then
 $x_*=x^*$.

\item[(iii)] If $\theta (x_0) = 0$, then $x_0 = x_* = x^*$.
\end{lem}

\begin{proof}
Let $x_*=\min \{x \in [a_1,a_2] \ : \ \theta (x) \pa_x\theta(x)=0 \}$ and $x^*=\max \{x \in [a_1,a_2] \ : \ \theta (x) \pa_x\theta(x)=0 \}.$
The existence of $x_*$ and $x^*$ follows from (\ref{e bc a-})-(\ref{e bc a+}) and $\theta  \pa_x\theta \in C [a_1,a_2]$.
Since
\[
\pa_x (\theta\pa_x\theta)=(\pa_x\theta )^2+\theta\pa_x^2\theta=
(\pa_x\theta)^2+\beta^2B\theta^2\geq0 \ \ \text{a.e. on }(a_1,a_2),
\]
we obtain (i). The equality 
\[
 0= \theta (x^*) \pa_x \theta (x^*) -  \theta (x_*) \pa_x \theta (x_*) = \int_{x_*}^{x^*} \pa_x (\theta \pa_x \theta)\ \dd x=
\int_{x_*}^{x^*} [(\pa_x \theta )^2 + \beta^2 B \theta^2] \dd x
\]
implies $\pa_x \theta(x)=0$ and $B(x)=0$ a.e. on $[x_*,x^*]$.
When $x_*<x^*$, this and $\theta \not \equiv 0$ yield statement (ii). Statement (iii) follows from (ii).
\end{proof}

\subsection{Proof of Theorem \ref{t char Si nl}}
\label{ss proof of Si nl}

Let $\alpha=\re \om \neq 0$. Due to Lemma \ref{l prop Si} (ii) and (\ref{e Snl sym}), 
without loss of generality one can assume $\alpha>0$.

Suppose $\om \in \Si(B_0)$ for certain $B_0 \in \A$.
Put
\[
 \ \  C_0=\frac{\partial_x\theta(a_2,\om,B_0)}{ \om^2} .
\]
By (\ref{e paBD F}), 
\[
C_0 \frac{\pa F (\om;B_0) }{\pa B}(V)=\int_{a_1}^{a_2}
\theta^2 (s,\om;B_0) \ V (s) \ \dd s \ .
\]

Define the set $S_0 \subset \CC$ by
\begin{equation} \label{S_0}
 S_0 :=C_0 \frac{\pa F (\om;B_0)}{\pa B} [\A - B_0] =
\left\{ \int_{a_1}^{a_2}
\theta^2 (s,\om;B_0) \ V (s) \ \dd s \  \ : \ V \in \A-B_0 \right\}.
\end{equation}

The definition of $K (\om,B_0;V)$ in (\ref{e Om asy}) implies the equivalence
\begin{equation} \label{e cone <=> cone}
\cone S_0 \neq \CC  \quad \Longleftrightarrow \quad \cone K (\om, B_0; \A - B_0) \neq \CC .
\end{equation}

\textbf{Proof of implication (i) $\Rightarrow $(ii).} Let $y$ be an eigenfunction of (\ref{e nonlin}), (\ref{e bc a-}), (\ref{e bc a+}) associated with $\om$.
Let us define $B_0$ by equality (\ref{e B0=b1b2}) for all $x \in [a_1,a_2]$.
Then $\om \in \Si (B_0)$. Moreover, there exist $C \in \RR \setminus \{ 0 \}$ and
$\xi_* \in [-\pi,\pi)$ such that $y (x) = C e^{-\ii \xi_* /2} \theta (x, \om;B_0)$.
It follows from (\ref{e B0=b1b2}) that
\begin{equation} \label{e B*=b1b2}
B_0 (x) = \left\{ \begin{array}{cc}
b_1 (x), & \text{ when \ \  $\theta^2 (x) \in e^{i\xi_*}\overline{\CC_-}$} \\
b_2 (x), & \text{ when \ \ $\theta^2 (x) \in e^{i\xi_*} \CC_+$}
\end{array} \right. .
%\text{for a.a. $x \in [a_1,a_2]$}.
\end{equation}
This implies that for every $V \in \A-B_0$ the following assertion hold:
\begin{eqnarray*}
V (x) \ge 0 & \quad & \text{ for a.a. $x$ such that } \theta^2 (x) \in e^{i\xi_*}\overline{\CC_-} , \\
V (x) \le 0 & \quad & \text{ for a.a. $x$ such that } \theta^2 (x) \in e^{i\xi_*}\CC_+ .
\end{eqnarray*}
So
$\int_{a_1}^{a_2} \theta^2 V \ \dd s \in  e^{i\xi_*}\overline{\CC_-} $ for each $V \in \A-B_0$.
Thus, $\cone S_0 \subset  e^{i\xi_*}\overline{\CC_-}$. Equivalence  (\ref{e cone <=> cone}) concludes the proof of (ii).

\textbf{Proof of implication (ii) $\Rightarrow $(i).}
Put $\arg_* \theta^2 (x) := 2 \arg_* \theta (x)$.
Define the sets $E^\pm_s \subset [a_1,a_2] $ by
$$
E^+_s : = \supp [b_2-B_0]_+ , \qquad
 E^-_s := \supp [B_0-b_1]_+ ,
 %\qquad E_+ : = \{ x : b_2(x) > B_0 (x) \}, \qquad E_- := \{ x : B_0(x) > b_1 (x) \}
$$
(for  $[\cdot ]_+$ see the notation part of Section \ref{s Intro}).

\begin{lem} \label{l th Epm}
Let $\re \omega > 0$ and $\theta (x) = \theta (x,\om;B_0)$. Then
$$ \theta^2 [E_s^+ ] \cup \left( -\theta^2 [E_s^-] \right) \subset \ \ \overline{\cone S_0},$$
%there exists $\xi_*\in [0,2\pi)$ such that
%$$ \theta^2 [E^+_s\setminus \{x_0\}] \cup \left( -\theta^2 [E^-_s\setminus\{x_0\}] \right) \subset \ \ \overline{\cone S_0},$$
i.e., the image of the set $E^+_s $ under the function $\theta^2$ and
the image of the set $E^-_s$ under the function $(-\theta^2)$ are subsets of the closed cone $\overline{\cone S_0}$
%$e^{i\xi_*}\overline{\CC_+}$.
\end{lem}

\begin{proof}Recall that the point $x_0$ was defined in
Lemma \ref{l turn int} and that, in the case when $x_0$ does not exist, the notation $\{ x_0 \}$ means the empty set.

\emph{Step 1.} If $x_0$ exists, $\theta (x_0) = 0 \in \overline{\cone S_0}$.

\emph{Step 2.} Suppose $x_1 \in  E^+_s \setminus \{ x_0 \}.$
Then $V_{\varepsilon}^+ (x) :=\chi_{(x_1-\varepsilon,x_1+\varepsilon)} (x) \ [b_2-B_0]_+ (x)$ has a nonempty support for every $\ep>0$.
Let us show that (for a suitable choice of values of $\arg$)
\begin{equation} \label{arg pr1-arg t*}
\arg \left[ C_0 \frac{\pa F (\om;B_0)}{\pa B} (V_{\varepsilon}^+) \right] - \arg_* \theta^2 (x_1) \rightarrow 0 \quad
\text{ as } \varepsilon \rightarrow 0 .
\end{equation}

By Lemma \ref{l turn int} (iv), %$\arg_* \theta^2 $ is defined at $x_1$ and, moreover,
\begin{equation} \label{e arg th x1}
\text{$\arg_* \theta^2 $ is defined and continuous
%and monotone (i.e., non-decreasing or non-increasing)
in a vicinity of $x_1$.}
\end{equation}
For small enough $\ep>0$,
$|\arg_* \theta^2 (x) - \arg_* \theta^2 (x_1)| < \pi/2$ in the interval $x \in (x_1-\ep,x_1+\ep)$ and therefore
the equality
\begin{equation*} \label{v1 v S-0}
C_0 \frac{\pa F}{\pa B} (V_{\varepsilon}^+)  =
\int_{a_1}^{a_2} \theta^2  \ V_{\varepsilon}^+  \ \dd s= \int_{x_1-\varepsilon}^{x_1+\varepsilon}
\theta^2  [b_2-B_0]_+  \dd s,
\end{equation*}
implies that
$C_0 \frac{\pa F}{\pa B} (V_{\varepsilon}^+) \in \overline{\cone \theta^2 [(x_1-\varepsilon,x_1+\varepsilon)] } \ \setminus \ \{0\}$,
where $\theta^2 [(x_1-\varepsilon,x_1+\varepsilon)] $ is the image of $(x_1-\varepsilon,x_1+\varepsilon)$
under $\theta^2$. This and (\ref{e arg th x1}) yields (\ref{arg pr1-arg t*}).

\emph{Step 3.}
Suppose $x_1 \in  E^-_s \setminus \{ x_0 \}.$ Then, similar to Step 2,
$V_{\varepsilon}^- (x):= - \chi_{(x_1-\varepsilon,x_1+\varepsilon)}(x) [B_0-b_1]_+ (x)$ has a nonempty support for small $\ep$
and
%\begin{equation} \label{arg pr1-arg t* -}
$\arg \left[ C_0 \frac{\pa F }{\pa B} (V_{\varepsilon}^-) \right] - \arg_* \theta^2 (x_1)  - \pi \rightarrow 0 $
as $\varepsilon \rightarrow 0 $.
%\end{equation}
In other words,
\[
\arg C_0 \frac{\pa F}{\pa B} (V_{\varepsilon})- \arg [ -\theta^2 (x_1) ] \rightarrow 0
\quad
\text{ as } \varepsilon \rightarrow 0 .
\]
%for a suitable choice of values of $\arg$.

\emph{Step 4}. Since $B+V_{\varepsilon}^\pm \in \A$, we see from Steps 2-3 that
$\theta^2(x_1) \in \overline{\cone S_0}$ whenever $x_1 \in  E^+_s \setminus \{ x_0 \} $,
and $[-\theta^2(x_1)] \in \overline{\cone S_0} $ whenever $x_1 \in  E^-_s \setminus \{ x_0 \} $.
This completes the proof.
\end{proof}

\begin{lem} \label{l B via hp}
Let $\re \omega > 0$. Assume that $\cone S_0$ is contained
in a certain closed half-plane $e^{i\xi_*}\overline{\CC_-}$, where $\xi_* \in [-\pi,\pi)$.
Then:
\item[(i)] equality (\ref{e B*=b1b2}) holds for a.a. $x \in [a_1,a_2]$,
\item[(ii)] $\om \in \Si^\nl$ and $y = e^{-i\xi_*/2} \theta$
is an associated eigenfunction of (\ref{e nonlin}), (\ref{e bc a-}), (\ref{e bc a+}).
\end{lem}

\begin{proof}
\textbf{(i)} Recall that the set $E$ is defined by (\ref{e setE}).
From the definition of $E$ and $B_0 \in \A$, we see that
$E \subset X_0 \cup E^+_s \cup E^-_s  $, where $X_0$ is of zero measure.
Clearly, $\overline{\cone S_0} \subset  e^{i\xi_*}\overline{\CC_-}$.
By Lemma \ref{l th Epm},
\begin{equation} \label{e thEpm halfpl}
\text{$\theta^2 [E^+_s] \subset e^{i\xi_*}\overline{\CC_-} $ \qquad and \qquad
$\theta^2 [E^-_s] \subset e^{i\xi_*}\overline{\CC_+} $.}
\end{equation}
So $E^+_s \cap E^-_s \subset \{x : \theta^2 (x) \in e^{i\xi_*} \RR \}$. By Lemma \ref{l turn int} (iv)-(v),
\begin{equation} \label{e X1}
\text{the set \  
$X_1 := \   \{x : \theta (x) \in e^{i\xi_*} \RR \} \ \setminus \  [x_*,x^*]$ \ is at most finite.}
\end{equation}
Thus, $E \subset X_0 \cup X_1 \cup [x_*,x^*] \cup (E^+_s \setminus E^-_s) \cup
(E^-_s \setminus E^+_s)$.
On the other side, $B_0 (x) = b_1 (x) = 0$ a.e. on $[x_*,x^*] $ due to Lemma \ref{l turn int} (ii).
So $X_2 := ([x_*,x^*] \cap E) \setminus (E^+_s \setminus E^-_s) $ is of zero measure.

Summarizing, we see that the interval $[a_1,a_2]$ can be written as
\[
[a_1,a_2] = ([a_1,a_2] \setminus E) \cup X_0 \cup X_1 \cup X_2 \cup (E^+_s \setminus E^-_s) \cup
(E^-_s \setminus E^+_s) .
\]
On each of the sets of positive measure in the right side, equality (\ref{e B*=b1b2}) is fulfilled for a.a. $x$.

Additional explanations are needed for fact that $\theta^2 (x) \in e^{i\xi_*} \CC_+$ for a.a. $x \in E^-_s \setminus E^+_s$.
Assume that $x_1$ belongs to the set $X_3$ consisting of $x$ such that
$x \in E^-_s \setminus E^+_s$ and $\theta^2 (x)  \not \in e^{i\xi_*} \CC_+$.
By (\ref{e thEpm halfpl}), $ \theta^2 (x_1) \in e^{i\xi_*} \RR $.
By (\ref{e X1}), $X_3 \subset X_1 \cup [x_*,x^*]$. Assume that $\meas X_3 > 0$. Then $\meas (X_3 \cap [x_*,x^*])>0 $.
This yields that $\meas \left[ \supp [B-b_1]_+ \cap (x_*,x^*) \right] >0 $. The latter contradicts Lemma \ref{l turn int} (ii).
This contradiction shows that $\meas X_3 = 0$.

\textbf{(ii)}  Equality (\ref{e B*=b1b2}) implies (\ref{e B0=b1b2}). The latter yields (ii).
\end{proof}

Equivalence (\ref{e cone <=> cone}) and Lemmata \ref{l th Epm}, \ref{l B via hp} imply statement (i) of Theorem \ref{t char Si nl}.

\textbf{The last assertion of Theorem \ref{t char Si nl}} follows from 
Lemma \ref{l B via hp} (ii) and its proof.

\section{Local extremizers and optimization for a fixed frequency}
\label{s opt problem}

\subsection{Definitions of various optimizers}
\label{ss def optim}

We say that $\alpha \in \RR$ is \emph{an admissible frequency} if $\alpha = \re \om$ for some admissible quasi-eigenvalue $\om$. So $\re \Si [\A]$ is the set of admissible frequencies. Its properties are considered in Subsection \ref{ss ex A} and Appendix \ref{a adm fr},
where it is proved, in particular, that high enough frequencies are admissible at least for most popular settings of quasi-eigenvalue optimization problem.

\begin{defn}[\cite{Ka14}]
Let $\alpha$ be an admissible frequency.

\item[(i)] \emph{The minimal decay rate $\beta_{\min} (\alpha)$ for the frequency} $\alpha$
is defined by
\[
\beta_{\min} (\alpha) := \ \inf \{ \beta \in \RR \ : \ \alpha - \ii \beta  \in \Si [\A] \} .
\]

\item[(ii)] If $\om = \alpha - \ii \beta_{\min} (\alpha)$ is a quasi-eigenvalue for a certain admissible structure
$B \in \A$ (i.e., the minimum is achieved), we say that $\om $ and $B$ are of \emph{minimal decay for the frequency} $\alpha $.
\end{defn}

\begin{figure}[h]
    \centering
\includegraphics[width=\linewidth]{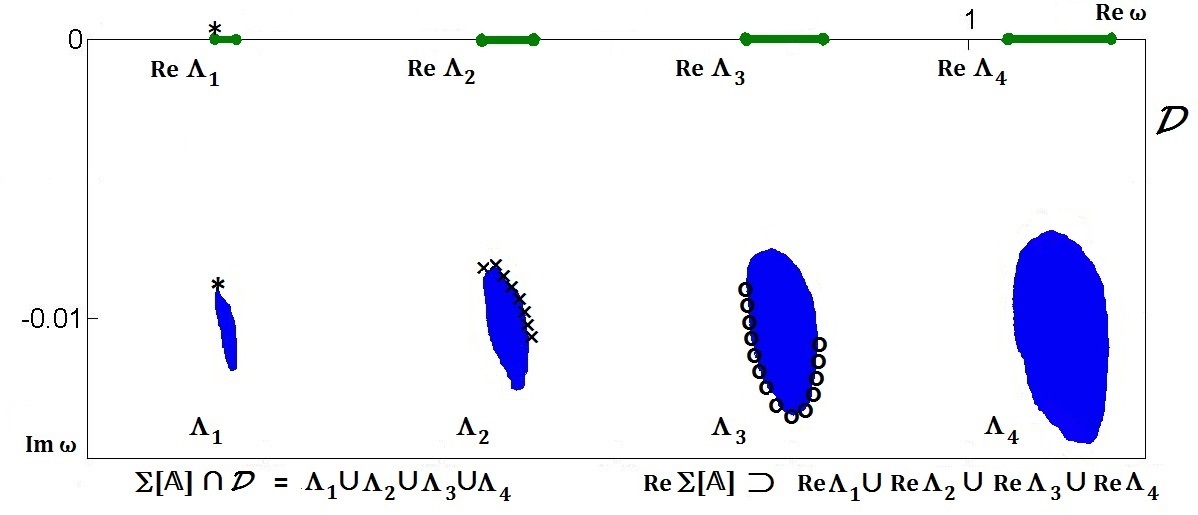}
    \caption{A schematic drawing (see explanations in Remark 
    \ref{r drawing}) of the part of the set of
admissible quasi-eigenvalues 
$\Si [\A]$ in the rectangle $\D = \{\re z \in (0, 1.2) , \ \im z \in (-0.015, 0) \}$ 
and of intervals of admissible frequencies generated by $\Si [\A] \cap \D$
for the following choice of parameters:   
$\nu_1 = 1$, $\nu_2 =+\infty$, $b_1 = 90$, $b_2 = 110$, $a_2 - a_1 = 1$. 
\newline
'*' \ \ mark one of local minimizer for $\beta_{\min}$ and the corresponding quasi-eigenvalue; \newline
'$\times$' \  mark quasi-eigenvalues of minimal decay for frequencies $\alpha$ in $\re \La_2$; \newline
'o' \ \ mark $\om$-parts of local minimizers of $\im \om$ for frequencies $\alpha$  in $\re \La_3$.
}
    \label{f drawing}
\end{figure}

A pair $(\om,B) \in \CC \times L^1$ is called admissible if $B\in \A$ and $\om \in \Si (B)$.

\begin{defn}
We say that an admissible pair $(\om_0 , B_0 ) $ is \emph{a local maximizer of $\im \om $ for the frequency $\re \om_0$} if there exist $\ep >0$ such that
$\im \om_1  \le \im \om_0 $ for any admissible
pair $(\om_1 , B_1 )$ satisfying the following three conditions
\begin{equation*}
\re \om_1  = \re \om_0, \ \ \
|\om_1 - \om_0 | < \ep , \ \ \
\| B_1 - B_0 \|_1  <  \ep .
\end{equation*}
\end{defn}

\emph{Local minimizers of $\im \om$} for a particular frequency are defined in a similar way.
A pair $(\om,B) $ is a \emph{local extremizer} if it is a local minimizer  or a local maximizer.

An admissible frequency $\alpha_0$ is \emph{a local minimizer for} $\beta_{\min}$ if there exists
$\ep >0$ such that $\beta_{\min} (\alpha_0) \le \beta_{\min} (\alpha) $ for all admissible
$\alpha $ in $ (\alpha_0 - \ep, \alpha_0 + \ep)$.
If for certain $\ep>0$ and all admissible $\alpha$ from a punctured
neighborhood $ (\alpha_0 - \ep, \alpha_0) \cup (\alpha_0, \alpha_0 + \ep) $ the strict inequality $\beta_{\min} (\alpha_0) < \beta_{\min} (\alpha) $ holds, $\alpha_0$ is said to be \emph{a strict local minimizer}.

Usually, in applied and numerical literature, local minimizers of $\im \om$ are considered without frequency restrictions. It seems that such approach goes back to %semi-rigorous
the definition of \cite{HS86}. Below we give a rigorous adaptation of that definition suitable for the optimization over $\A$.

Let us define on the set of admissible pairs $(\om,B)$ the decay rate functional $\Dr (\om,B) := - \im \om$.

\begin{defn} An admissible pair $(\om_0,B_0)$ is said to be a \emph{local minimizer for} $\Dr$ if
there exist $\ep>0$ such that $\Dr (\om_0,B_0) \le \Dr (\om_1,B_1)$ for every admissible pair $(\om_1,B_1)$ 
with $|\om_1 - \om_0| < \ep$  and $\| B_1 - B_0 \|_1 < \ep$.
Local maximizers for $\Dr$ are defined similarly.
\end{defn}

\begin{rem} \label{r loc opt}
Clearly, if $(\om_0,B_0)$ is a local minimizer (maximizer) for $\Dr$, 
then  $(\om_0,B_0)$ is a local maximizer (resp., minimizer) of $\im \om$ for the frequency $\re \om_0$.
\end{rem}

\begin{rem} \label{r drawing}
Figure \ref{f drawing} is drawn to illustrate the above definitions of various optimizers. 
It also provide a rough impression about a part of $\Si [\A]$ 
in $\D$
%$ = \{z \in \CC: \re z \in (0, 1.2) , \ \im z \in (-0.015, 0) \}$
 for 
the one-side open case $\nu_1 = 1$, $\nu_2 =+\infty$ (cf. Example \ref{ex 1-side D}).
% and the constant constraint functions  $b_1 = 90$, $b_2 = 110$. 
%The length $a_2 - a_1$ of the resonator equals $1$.
 The drawing is partially based on the MATLAB computations of the nonlinear spectrum $\Si^\nl$
 (see Section \ref{s NumEx} and Figure \ref{f Snl}) and partially on Proposition \ref{p Sclosed}, Theorems \ref{t small diel contrast}, \ref{t sm contr gen}, \ref{t nl ep}, and formula 
(\ref{e b min =}).

In agreement with Theorem \ref{t sm contr gen}, 
$\Si [\A] \cap \D $ consists of four `clouds'  $\La_1$, \dots, $\La_4$. These `clouds' are connected components of $\Si [\A]$ produced by all possible perturbations (inside $\A$) of quasi-eigenvalues $\om_{0,1} = \om_0 (b_2)$, $\om_{0,2} = \om_1 (b_2)$, \dots, $\om_{0,4} = \om_3 (b_2)$, resp., which are generated by the homogeneous structure $B_0  \equiv b_2$, see (\ref{e om n hom}) and Section \ref{s small cont}. 

The projections of these components to the real line $\re \La_j$  are subsets of the set of admissible frequencies $\re \Si [\A]$. However, 
we \emph{do not know whether $\bigcup_{j=1}^4 \re \La_j $ are all admissible frequencies in the interval $(0, 1.2)$}. Indeed, it is not known if there exist 
any admissible quasi-eigenvalues $\om$ with $\re \om \in (0, 1.2)$ lying outside of $\D$. 
One of the difficulties in the study of $\om \in \Si [\A]$ with large $|\om|$ is that, globally in $\CC$, $\Si (B)$ is not a continuous set-valued function  of $B$, see Remark \ref{r glob discont}. 

In these setting, the frequency $0$ is not admissible, see Remark \ref{r const b12 0 is not admis}.
\end{rem}

\subsection{Admissible frequencies and examples of admissible families}
\label{ss ex A}

Recall that the notation $ \lceil x \rceil $ ($\lfloor x \rfloor$) stands for the ceiling (floor) function.

In this subsection we consider several most reasonable and popular admissible families for 
which the choices of the constraints has a Physics motivation. 
For all this families we provide a quantitative version of the statement that 
\emph{high enough frequencies are admissible}. These estimates on the ranges of admissible 
frequencies are quite technical. They are derived in Appendix \ref{a adm fr} 
together with the proof that high enough frequencies are admissible in more general situation, 
which is described Proposition \ref{p adm fr} (iii). This proposition also considers other properties of the set 
$\re \Si [\A]$ of admissible frequencies and, roughly speaking, shows that \emph{the set  $\re \Si [\A]$ 
cannot have some 'wild'  structure}.

\begin{ex} \label{ex 2-side}
Let the constraint functions $b_{1,2}$ be constant and $0<b_1 < b_2$. Assume that $\nu_1 =
 \nu_2 =:\nu > 0$ and $\nu^2$ equals either $b_1$, or $b_2$. These settings represent the admissible family associated with two side open optical cavities of length $a_2-a_1$ consisting of transparent dielectrics with permittivities in the range from $b_1$ to $b_2$ 
\cite{SSH09}. 
At the endpoints $a_{1,2}$, the cavity is connected to half-infinite homogeneous outer media consisting of the dielectric material with permittivity $\nu^2$.
In Optical Engineering modeling, the outer medium is 
usually the same as one of  dielectrics with extreme allowed permittivities 
(usually, either  $b_1 = \nu^2 =1$ represents vacuum, or $b_2=\nu^2 $ corresponds to a cavity connected to a waveguide of permittivity $b_2$ \cite{BPSchZsch11}). Under this choice  of parameters, formulae
(\ref{e re Si A out}) and (\ref{e re Si A in and out =}) of Appendix \ref{a adm fr} show that 
\[
\text{all frequencies in the range }
| \alpha | > 
\frac{\pi }{(a_2-a_1) b_2^{1/2} } 
\left\lceil  \frac{ b_1^{1/2} }{ b_2^{1/2} -b_1^{1/2} } \right\rceil
\text{ are admissible.}
\]
\end{ex}

\begin{ex} \label{ex 1-side D}
Let the constraints $b_{1,2}$ be constants and $0<  b_1 < b_2$. Assume that $\nu_1^2 = b_1$ and $ \nu_2 = \infty$.
One of Optics interpretations of these settings is one-side open cavity, which,
in particular, is used for the laser modeling \cite{U75}.
 The Dirichlet boundary condition $y(a_2) = 0$ generated by the parameter $ \nu_2 = \infty$ corresponds to a perfect mirror (a plate of a perfect electric conductor perpendicular to the x-axis). So EM waves are reflected without dissipation from the interface plane $x=a_2$. Through 
the interface plane $x=a_1$ the waves radiate into homogeneous half-infinite medium with permittivity $\nu_1^2$. It was proved in \cite{Ka13_Opt} that 
\begin{equation} \label{e adm a met}
\text{all frequencies in the range }
| \alpha | \ge \frac{\pi }{(a_2-a_1) b_2^{1/2} } 
\left(\frac12 + \left\lceil  \frac{ b_1^{1/2} }{ b_2^{1/2} -b_1^{1/2} }  - \frac12 \right\rceil \right)
\text{ are admissible.}
\end{equation}
\end{ex}

\begin{ex} \label{ex 1-side N}
In Mechanics models involving the equation for transverse oscillations of a nonhomogeneous string, the linear density $B(x)$ of the string on some intervals may be much less than on others. In such cases,
the massless string approximation is employed. This means that $B(x)$ 
(and so the lower constraint $b_1(x)$) can be equal to $0$ on certain intervals. 
Keeping this in mind, consider constant constraints $b_{1,2} \in [0,+\infty)$ satisfying 
$0\le b_1 < b_2$. Assume that $\nu_1 = 0$ and $ 0< \nu_2 < +\infty$. 
The choice $\nu_1= 0 $ leads to the Neumann boundary condition $y'(a_1) = 0$ and 
corresponds to the assumption that the left end of the string is "free" 
(i.e., a massless ring at the end $a_1$ of the string is sliding without friction on a pole pointing 
in the transverse direction). 
The Mechanics interpretation of the condition (\ref{e bc a+}) at $a_2$ is 
the linear damping condition (which arises, for instance, if a massless plate immersed 
in a fluid is connected to the right end $a_2$). In such settings, $\nu_2 $ is 
the damping coefficient (if the units are chosen such that the tension of the string equals $1$). 
In this context, $\nu_2$ has no natural connection with extreme allowed densities of the string. 
The estimates of Appendix \ref{a adm fr} shows that for this model
\[
\text{all frequencies in the range $ \displaystyle | \alpha | \ge \frac{3\pi }{2(a_2-a_1) ( b_2^{1/2} -b_1^{1/2} ) } $  
are admissible.}
\]
\end{ex}

Another interpretation of the Neumann condition $y'(a_1) = 0$, which makes also sense from Optics point of view, is given in the next example.

\begin{ex} \label{ex sym}
Let us assume additionally to the settings of Example \ref{ex 2-side} that $a_2=-a_1$ and let us restrict optimization to structures symmetric w.r.t. origin, i.e., to $B$ such that $B(x)=B(-x)$ and $b_1 \le B \le b_2 $ a.e.. This model of a symmetric resonator is used often to reduce computational costs, see e.g. \cite{KS08}.
Since $B$ is an even function and $\nu_1 = \nu_2$, it is easy to see that modes $y$ of such a resonator are either even, or odd, and so either condition $y'(0) = 0$, or the condition $y(0) = 0$ is fulfilled for them. Thus, the quasi-eigenvalue problem and the related optimization problem can be essentially reduced to two: the problem considered in Example \ref{ex 1-side N} and a problem similar to that of Example \ref{ex 1-side D} (possibly, with different $\nu_1$). 
\end{ex}

%\subsection{Properties of the set of admissible frequencies}
%\label{ss isol adm fr}

Denote $b_1^{\max} := \esssup b_1 (x) $ and $b_2^{\min}:=\essinf b_2 (x) $.
Recall that intervals are, by definition, connected subsets of $\RR$ and that an interval called \emph{degenerate} if it consists of one point.

\begin{prop}\label{p adm fr}
$(i)$ If $I$ is a connected component of $\re \Si [\A]$, then either $I$ is a nondegenerate interval, or $I = \{0\}$. In particular, the only possible isolated admissible frequency is $0$.

\item[(ii)] Suppose $b_1^{\max} \le b_2^{\min} $. If $b_1^{\max} < \nu_1$ or $b_2^{\min} > \nu_2$, then $0$ is an admissible frequency.

\item[(iii)]
Suppose $b_1^{\max} < b_2^{\min} $. Then high enough frequencies are admissible.
Moreover,
\begin{equation} \label{e re Si superset}
\re \Si [\A] \supset \left\{ | \alpha | \ge  
\frac{2 \pi } {(a_2-a_1) ( \sqrt{b_2^{\min}}-\sqrt{b_1^{\max}}) } \right\}
\end{equation}
\end{prop}

\begin{proof} 
\textbf{(i)}
It is enough to show that if $\om_0 \in \Si (B_0)$, $B_0 \in \A$, and $\re \om_0 \neq 0$, then
$B_0$ can be piecewise linearly perturbed in $\A$ such that perturbations of $\om_0$ 
do not preserve $\re \om_0$. 

Taking nontrivial $V \in \A - B_0$ consider perturbations $B_0 + \zeta V$ with sufficiently small $\zeta>0$. 
If for perturbed quasi-eigenvalues $\om (\zeta)$ the real part is preserved, we fix a certain branch $\om (\zeta)$ and certain sufficiently 
small $\zeta_1 \in (0,1)$ such that $B_1 = B_0 +\zeta_1 V$ is not an extreme point of $\A$. 
By Theorem \ref{t char Si nl} and the remarks before it, $\cone K(\om(\zeta_1), B_1; \A - B_1) = \CC$. By Proposition \ref{p per k mult}, 
$\om (\zeta_1)$ can be perturbed with change of the real part (see also Proposition \ref{p int points}).

\textbf{(ii)} follows from Proposition \ref{p const struct} (ii).

\textbf{(iii)} For the settings of Example \ref{ex 1-side D}, 
(\ref{e re Si superset})  follows easily from formula (\ref{e adm a met}), which was proved in \cite{Ka13_Opt} in slightly different form. The proof in general settings can be obtained using the same idea, but requires more technical considerations and involves a number of different cases of mutual arrangements of $\nu_{1,2}$, $b_{1}^{\max}$, and  $b_2^{\min}$. This proof is given in Appendix \ref{a adm fr}.
\end{proof}

\subsection{Existence of optimizers  and properties of $\Si [\A]$}
\label{s exist mindec}

\begin{prop} \label{p Sclosed}
The set $\Si [\A] $ is closed.
\end{prop}

\begin{proof}
Since $\A \subset L^1$, the proof requires some modifications of standard compactness arguments. 
% (confer e.g. \cite{Ka13}).
For each $B \in \A$ there exists unique $H (x) = H (x;B) $ such that
\begin{gather*}
B (x) = (b_2 (x) -b_1 (x)) H (x) + b_1 (x) \text{ a.e. on } (a_1,a_2) \label{e H}
%\\
%\text{, \ \  $0 \le H(x) \le 1$ a.e. on $(a_1,a_2)$, and $H(x) = 0$ a.e. on $(a_1,a_2) \setminus E$. }
%\label{e H set}
\end{gather*}
and $H (x)$ belongs to the set $\A_* $ of $L^\infty$-functions satisfying
\begin{gather*}
\text{$0 \le H(x) \le 1$ a.e. on $(a_1,a_2)$ and $H(x) = 0$ a.e. on $(a_1,a_2) \setminus E$.}
\end{gather*}

Let us take $R>\|b_2\|_1$. One can see that
\begin{multline} \label{e H cont}
\text{the map $H (\cdot; B) \mapsto B (\cdot)$ is a bijection of $\A_*$ to $\A$ continuous } \\ \text{
from the weak$^*$ topology of $L^\infty$ to the metric topology of $(\BB_R,\rho)$}
\end{multline}
(see Subsection \ref{ss weak cont}, the continuity is sequential and w.r.t. the corresponding induced topologies on $\A_*$ and $\A$).
Lemma \ref{l weak cont} (i) and the sequential Banach-Alaoglu theorem applied to $\A_*$ complete the proof.
\end{proof}

Proposition \ref{p Sclosed} and $\Si [\A] \subset \CC_-$ imply three following results.

\begin{cor} \label{c exist}
(i) For each admissible frequency $\alpha$, there exist a structure $B$ and the quasi-eigenvalue $\om$ of minimal decay, i.e., 
$\om = \alpha - \ii \beta_{\min} (\alpha) $ belongs to $\Si (B)$ for certain $B \in \A$.

\item[(ii)] The pair  $(\om , B )$ from statement (i) is a local maximizer of $\im \om$ for the frequency $\alpha$.
\end{cor}

\begin{prop} \label{p lower semicont}
The function $\beta_{\min}$ is lower semicontinuous, i.e., if $\{\alpha_n \}_0^{\infty} \subset \re \Si [\A]$ and $\alpha_n \to \alpha_0$, then $\beta_{\min} (\alpha_0) \le
\lim \inf \beta_{\min} (\alpha_n)$.
\end{prop}

\begin{prop} \label{p exist loc}
If an admissible frequency $\alpha$ is a local minimizer for $\beta_{\min}$, then the pair $(\om,B)$ from Corollary \ref{c exist} is a local minimizer for $\Dr$.
\end{prop}

The existence of local minimizers for $\beta_{\min}$ and $\Dr$ is considered in next two sections.

\section{Non-uniqueness and local minima for small $\alpha$}
\label{s 0fr}

Pure imaginary quasi-eigenvalues $\om \in \ii \RR$ are the subject of special interest in Mechanics, see 
e.g. \cite{CZ94,CO96}. They corresponds to overdamped and critically damped oscillations.

\begin{prop} \label{p loc extr al=0}
If $\re \om_0 = 0$ and $(\om_0,B_0)$ is a local extremizer of $\im \om$ for the frequency $0$, then either $B_0 \equiv b_1$, or $B_0 \equiv b_2$. In particular, if $B_0$
is a structure of minimal decay for the frequency $0$, then  either $B_0 \equiv b_1$, or $B_0 \equiv b_2$.
\end{prop}

\begin{proof}
Assume that $(\om_0,B_0)$ is a local extremizer of $\im \om$ for frequency $0$
and that $B_0 \neq b_1$, $B_0 \neq b_2$ (in the sense of $L^1$-space). Then
it follows from (\ref{e paBD F}) and Lemma \ref{arg th with alf=0},
that there exist $V_1^\pm \in \A -B_0$ such that $\pm \frac{\pa F (\om_0;B_0)}{\pa B} (V^\pm_1) >0$.
This,  Proposition \ref{p per k mult}, and the symmetry of $\Sigma(B)$ w.r.t.  $\ii \RR$
imply that
there exist $V^\pm_2 \in \A -B_0$ and $\om_\pm (\zeta) \in \Si (B_0 + \zeta V^\pm_2)$ such that for small enough
$\zeta > 0$,
\[
\re \om_\pm (\zeta)  = 0 , \ \ \ \ \pm [ \im \om_\pm (\zeta) - \im \om_0 ] > 0 , \ \ \
\text{ and } \ \om_\pm (\zeta) \to \om_0 \text{ as } \zeta \to 0
\]
(the arguments are similar to \cite[Section 4.4]{Ka13}, note that $V^+_2$ and $V^-_2$ are not necessarily different).
Thus, $(\om_0,B_0)$ is not a local extremizer.
\end{proof}

\begin{thm} \label{t al=0 beta min}
Assume that $\alpha=0$ is an admissible frequency. Then
\item[(i)] $\Setim (b_1,b_2):= \ii \RR \cap [\Si (b_1) \cup \Si (b_2)]  $ is a nonempty closed set.
\item[(ii)] $\beta_{\min} (0) = \min\limits_{\om \in \Setim (b_1,b_2)} |\im \om| $
\item[(iii)] $\alpha = 0 $ is a strict local minimizer for $\beta_{\min}$ whenever
$\om = - \ii \beta_{min} (0)$ does not belong to $\Si_\mult (b_1) \cup \Si_\mult (b_2)$
(i.e., whenever the quasi-eigenvalue of minimal decay $\om$ is not a multiple quasi-eigenvalue of $b_1$ and of $b_2$).
\end{thm}

\begin{proof}
\textbf{(i) and (ii)} follow from the assumption $0 \in \re \Si [\A]$, Proposition \ref{p loc extr al=0}, and Proposition \ref{p Sclosed}.

If $0$ is an isolated point of $\re \Si [\A]$, statement \textbf{(iii)} follows from
the definition of a strict local minimizer for $\beta_{\min}$ (see Section \ref{ss def optim}). 
Let us denote by $I$ the connected component of $\re \Si [\A]$ containing $0$, 
and let us prove \textbf{(iii)} in the case when $I \neq \{ 0 \}$.

  By Lemma \ref{l prop Si} (ii) and Proposition \ref{p adm fr} (i), $I$ contains a nonempty symmetric interval $(-\alpha_0,\alpha_0)$. Consider
$\beta_* := \liminf_{\alpha \to 0-} \beta_{\min} (\alpha)$. Then also  $\beta_* = \liminf_{\alpha \to 0+} \beta_{\min} (\alpha)$
due to Lemma \ref{l prop Si} (ii).

If $\beta_* = +\infty$, it is obvious that $\alpha=0$ is a strict minimizer. Consider the case $\beta_* < +\infty$.

 Since $\Si [\A]$ is closed,
there exist sequences $\{B_n\}_{1}^{+\infty} \subset \A$ and $\om_{\pm n} \in \Si (B_n)$ such that
\begin{eqnarray*}
& \text{$\im \om_{\pm n} \to - \beta_*$,  \ \ \ $\re \om_{\pm n} \to 0 \pm 0$, \ \ \
and }\\
& \text{ the sequence $B_n$ converges to a certain $B_* \in \A$ w.r.t. the metric $\rho$
of Subsection \ref{ss weak cont}}
\end{eqnarray*}
 (this can be proved with the use of the map $H (\cdot; B) \mapsto B (\cdot)$
from (\ref{e H cont}) and the sequential Banach-Alaoglu theorem applied to $\A_*$).

Let us show that $\om_* := - \ii \beta_*$ is a multiple quasi-eigenvalue of $B_*$.  Indeed, $\om_* \in \Si (B_*)$ due to Lemma \ref{l weak cont} (i). One can take $\ep>0$ such that $\overline{\DD_\ep (\om_*)}$ does not contain other points of $\Si (B_*)$.
For large enough $n$, the disc $\DD_\ep (\om_*)$ contains at least two distinct quasi-eigenvalues $\om_{\pm n}$ of $B_n$. Hence, it follows from Proposition \ref{p weak cont mult} that the multiplicity of $\om_*$ as a quasi-eigenvalue of $B_*$ is $\ge 2$.

Now, the assumption that $\om = -\ii \beta_{\min} (0)$ is a simple quasi-eigenvalue of any corresponding structure $B$ of minimal decay (by Proposition \ref{p loc extr al=0},
$B$ is either $b_1$, or $b_2$) yields $\om \neq \om_* $ (since they have different multiplicities). So $\beta_* > \beta_{\min} (0)$. This completes the proofs of statement (iii).
\end{proof}

Now we will show that each of the statements (i)-(iii) of Theorem \ref{t al=0 beta min} takes place
for constant constraints $b_{1,2}$ under the assumption that at least one of the constants  $b_1^{1/2}$, $b_2^{1/2}$ is not in the interval
$[\nu_1,\nu_2]$.

Recall that the positive function $K_1 (b)$ defined in Proposition \ref{p const struct}
for positive $b$ is given by $ K_1 (b) = K_1 (b,\nu_1,\nu_2)
:=  \frac{1+\nu_1/\nu_2 } {\frac{\nu_1}{b^{1/2}} + \frac{b^{1/2}}{\nu_2}}  $.
Define
\begin{eqnarray*}
K_2 (b) & = & K_2 (b,\nu_1,\nu_2)  :=  \sgn [ 1-K_1 (b)]
\left| \frac{1-K_1(b)}{1+K_1(b)} \right|^{1/\sqrt{b}} \text{ for } b>0 , \\
K_2 (0) & = & K_2 (0,\nu_1,\nu_2)  := \exp( - 2/\nu_1 - 2/\nu_2)
\quad \text{(when $\nu_1 = 0$, we suppose $K_2 (0) := 0$).}
\end{eqnarray*}

\begin{cor} \label{c const b12 0 adm}
Let $b_1$ and $b_2$ be nonnegative constants such that $b_1^{1/2} < \nu_1$ or $b_2^{1/2} > \nu_2$.
Then:
\item[(i)] $\alpha = 0 $ is a strict local minimizer for $\beta_{\min}$ and
$\beta_{\min} (0) = - \frac{1}{2(a_2-a_1)} \, \ln \, \max \{K_2 (b_1), K_2 (b_2)\}$.
\item[(ii)] If $ K_2 (b_1) = K_2 (b_2)$, then the set of structures of minimal decay for the frequency $0$ consists
of two structures, $b_1$ and $b_2$.
\item[(iii)] If $ K_2 (b_1) > K_2 (b_2)$, then $b_1$  is the unique structure of minimal decay for the frequency $0$.
\item[(iv)] If $ K_2 (b_1) < K_2 (b_2) $, then $b_2$  is the unique structure of minimal decay for the frequency $0$.
\end{cor}

\begin{rem} \label{r mult const al=0}
By Proposition \ref{p const struct}, the quasi-eigenvalue of minimal decay 
$\om_0 = -\ii \beta_{\min} (0)$
has multiplicity 1 for every associated structure of minimal decay in all the cases of Corollary \ref{c const b12 0 adm}
(cf. \cite{CO96,Ka13_KN,Ka13,Ka14}).
\end{rem}

\begin{proof}[Proof of Corollary \ref{c const b12 0 adm}.]
Proposition  \ref{p const struct} implies that $0$ is an admissible frequency. It follows from
(\ref{e sign of 1-K1}) that $K_2 (b) \le 0$ exactly when $\Si (b) \cap \ii \RR = \varnothing$.  On the other side, $K_2 (b) > 0$ exactly when $\Si (b) \cap \ii \RR = \{ \om_0 (b) \}$, where $\om_0 (b)$ is defined by (\ref{e om n hom}). In this case, $K_2 (b) = \exp \left( 2 (a_2-a_1) \im \om_0 (b) \right)$. This, Theorem \ref{t al=0 beta min}, and Proposition \ref{p const struct} complete the proof
(see also Remark \ref{r mult const al=0}).
\end{proof}

\begin{ex}[non-uniqueness of optimizer] \label{ex nonuniq}
Let $C \in \RR_+$. Put $b_1 (x) \equiv C^2$, $b_2 (x) \equiv 4 C^2$, 
$\nu_1 = \nu_2 = C\sqrt{3}$. This 
correspond to optimization problem for a cavity constrained by $C^2 \le B(x) \le  4 C^2$
and placed into the outer medium 
$(-\infty,a_1) \cup (a_2,+\infty)$ with permittivity  $3C^2$ (cf. Example \ref{ex 2-side}).  
The length $a_2 - a_1$ is not important for the present example. 

Let $\om_0 := -  \frac{ \ii }{ 2 C (a_2 -a_1) }
\ln \left( 7+4\sqrt{3} \right)$.
Then the pairs $(\om_0,b_1)$ and $(\om_0,b_2)$ are:
\begin{itemize}
\item[(a)] local minimizers for $\Dr$ and local maximizers of $\im \om $ for the frequency $0$.
\end{itemize}
 Moreover,
\begin{itemize}
\item[(b)] $b_1$ and $b_2$ are structures of minimal decay for the frequency $0$,
\item[(c)] the frequency $\alpha_0 = 0$ is a strict local minimizer for $\beta_{\min}$.
\end{itemize}

These statements follow from Corollary \ref{c const b12 0 adm} (i)-(ii) and straightforward calculations. This example answers partially the question of uniqueness of optimal structure discussed in \cite{HS86,Ka14,Ka13_KN}
\end{ex}

\begin{rem}[cf. \cite{CZ95,Ka13}] \label{r const b12 0 is not admis}
Considering the case when constant functions $b_{1,2}$ do not satisfy the assumption of Corollary \ref{c const b12 0 adm}, we obtain the following result: if $\nu_1 \le \sqrt{B(x)} \le \nu_2$ a.e., then $\Si (B) \cap \ii \RR = \emptyset$. 
Indeed, Theorem \ref{t al=0 beta min} (i) and Proposition \ref{p const struct} imply that $0 \not \in \re \Si [\A]$ when $\nu_1 \le b_1^{1/2} \le b_2^{1/2} \le \nu_2$. 
\end{rem}

\section{Local minima in the case of small contrast}
\label{s small cont}

Local minimizers for $\beta_{\min}$ and for the decay rate functional $\Dr$ were defined in Section \ref{ss def optim}.

\begin{thm} \label{t small diel contrast}
Let $b  \in \RR_+ \setminus [\nu_1^2,\nu_2^2]$.
Then for every natural number $N $ there exists $\ep>0$ such that for any admissible family
$\A$  defined by constant constraints
\[
b_1 \equiv b - \ep_1 , \qquad b_2 \equiv b + \ep_2 \ \
\text{ with } 0<\ep_1+\ep_2 < \ep \text{ and }  \ep_{1,2} \ge 0,
\]
the following statements hold:
\item[(i)] In the interval $ |\alpha | < \frac{(N+1/2)\pi}{b^{1/2} (a_2 -a_1) }$
there exist at least $2N+1$  local minimizers for $\beta_{\min}$.

\item[(ii)] There exist at least $2N+1$  local minimizers and at least $2N+1$  local maximizers for
$\Dr$ with associated frequencies in the interval $ |\alpha | < \frac{(N+1/2)\pi}{b^{1/2} (a_2 -a_1) }$.
\end{thm}

\begin{rem} \label{r small diel contrast}
In the case $b  \in (\nu_1^2,\nu_2^2)$, Theorem \ref{t small diel contrast} becomes valid if one change the number of extremizers to $2N$ everywhere and the interval to $ |\alpha | < \frac{N\pi}{b^{1/2} (a_2 -a_1) } $ (see Figure \ref{f drawing} and Remark \ref{r drawing}).
\end{rem}

The theorem  follows from the next (more general) result, its proof, and Proposition \ref{p const struct}.

Recall that a set $S \subset \CC$ is called path-connected if for any two points $z_{1,2} \in S$ there exist a continuous function $f:[0,1] \to S$
such that $f(0) = z_1$ and $f(1)=z_2$.

\begin{thm} \label{t sm contr gen}
Let $B_0 \in L^1$, $B_0 (x) \ge 0$ a.e., and let $\D$ be an open bounded subset of $\CC$.
Assume that $\Si (B_0) \cap \D$ consists of $n \in \NN $ distinct points $\{ \om_{0,j}  \}_{j=1}^n$ and there are no points of $\Si (B_0)$ on the boundary $\Bd \D$.

\item[1)] Then there exists $\ep>0$ such that for any functions
$b_{1,2} (x)$ satisfying
\begin{equation} \label{e b1<B0<b2}
0 \le b_1 \le B_0 \le b_2 \ \text{ a.e. \ \ \   and } 0 < \| b_1 - b_2 \|_1 < \ep
\end{equation}
the following statements hold:
\subitem (i) $\Si [\A] \cap \D$ consists of $n$ (disjoint) closed path-connected components
$\La_1$, \dots, $\La_n$ satisfying $\om_{0,j} \in \La_j$, $j=1,\dots,n$.

\subitem (ii) There exist at least $n$ local minimizers $( \om_j^{\min} , B_j^{\min} )$ and at least $n$ local maximizers $( \om_j^{\max} , B_j^{\max} )$ for $\Dr$ such that $\om_j^{\min(\max)} \in \La_j$, $j=1$, \dots, $n$.

\subitem (iii) For each $\alpha \in \cup_1^n \re \La_j$, there exists at least one local minimizer and at least one local maximizer of $\im \om$ for the frequency $\alpha$.

\item[2)] If, additionally, $\D$ has the form $ \{ \alpha_1<\re z < \alpha_2 , \ \ - \beta_0<\im z < 0 \}$ with certain $\alpha_{1,2}, \beta_0 \in \RR$,  
and if $\ep$ is chosen as in statement (1), then in the frequency interval $(\alpha_1,\alpha_2)$
there exists at least one local minimizer $\alpha_0$ for $\beta_{\min}$.
This minimizer $\alpha_0$ can be chosen such that $\beta_{\min} (\alpha_0) = \min_{\om \in \Si [A] \cap \D} |\im \om|$.
\end{thm}

The rest of this section is devoted to the proof of Theorem \ref{t sm contr gen}.

\begin{prop} \label{p om path}
Assume that $B_1 , B_2 \in L^1_\CC$ and a bounded open set $\D \subset \CC$ are such that
for all $ t \in [0,1]$,
\begin{equation} \label{e as om not in bd}
\Si (B(t)) \cap \Bd \D = \varnothing , \qquad \text{ where } B(t) := (1-t)B_1 + t B_2 .
\end{equation}
Then there exist a nonnegative integer $n$ and a family $\{ \om_j (\cdot) \}_1^n$
of continuous complex-valued functions on $[0,1]$ such that for all $t \in [0,1]$,
\begin{equation*}
\Si (B(t)) \cap \D = \{ \om_j (t) \}_1^n  \qquad \text{ taking multiplicities into account}.
\end{equation*}
\end{prop}

\begin{proof} 
Since $F (\cdot; B ) \not \equiv 0$ for every $B \in  L^1_\CC$ (see Section \ref{ss F}),
the total multiplicity of quasi-eigenvalues of $B$ in $\D$ is finite.
It follows from assumption (\ref{e as om not in bd}),
Proposition \ref{p weak cont mult}, the Weierstrass preparation theorem, and the Puiseux series theory
(see e.g. \cite[Theorem XII.2]{RS78IV} and \cite{VT74,Ka13,Ka14}) that for every $s \in [0,1]$
there exist $\ep (s)>0$, a nonnegative integer $n_s$, and continuous in $I_s := (s-\ep(s),s+\ep(s))$ functions
$\Om_{s,j} (t)$, $j=1$, \dots, $n_s$, such that $\Si (B(t)) \cap \D = \{ \Om_{s,j} (t) \}_{j=1}^{n_{\scriptstyle s}} $
for all $t \in I_s$
(taking multiplicities into account).

It is easy to see that $n_s$ does not depend on $s$. Let us denote this number $n$.
Choosing a minimal finite subcover from the open cover $\bigcup_{t \in [0,1]} I_t$ of $[0,1]$,
it is easy also to construct (using $I_s$ and $\Om_{s,j} (t)$) a finite partition $0=t_0 < t_1 < \dots < t_k = 1$
of the interval $[0,1]$ and families $\{ \om_{i,j} (\cdot) \}_{j=1}^n$, $i=0$, \dots, $k-1$, of continuous on $[t_i,t_{i+1}]$
functions with the property that $\Si (B(t)) \cap \D = \{ \om_{i,j} (t) \}_{j=1}^n $ for every $t \in [t_i,t_{i+1}]$
(taking multiplicities into account).
Now one can continuously 'glue' these families at $t_1$, \dots, $t_{k-1}$ to produce the desired family
$\{ \om_j (\cdot) \}_1^n$.
\end{proof}

\begin{proof}[Proof of Theorem \ref{t sm contr gen}]
Let us take $\de>0$ such that the closed discs $\overline{\DD_\de (\om_{0,j})}$ are pairwise disjoint and all lie in $\D$.
By Proposition \ref{p weak cont mult}, choosing $R>\| B_0 \|_1$, we can ensure that there
exists a neighborhood $W$ of $B_0$ in the topology of the metric space
$(\BB_R , \rho )$  such that for all $B \in W$:
\begin{itemize}
\item[(a)] the total multiplicity of quasi-eigenvalues of $B$ in each of the open discs $\DD_\de (\om_{0,j})$
coincides with the multiplicity of $\om_{0,j}$ for $B_0$,
\item[(b)] the total multiplicity of quasi-eigenvalues of $B$ in $\D$
coincides with that of $B_0$,
\item[(c)] there no quasi-eigenvalues of $B$ on each of the boundaries $\Bd \D$, $\Bd \DD_\de (\om_{0,j}) $,
$j=1$, \dots, $n$.
\end{itemize}

Let us take $\ep>0$ such that $\BB_{\ep} (B_0) \subset W$ and take $b_{1,2}$ satisfying (\ref{e b1<B0<b2}).
Then $\A \subset W$. Denote $\La_j := \Si [\A] \cap \overline{\DD_\de (\om_{0,j})}$. Then $\La_j$ are closed. By properties (a) and (b),
$\Si [\A] \cap \D = \cup \La_j $. By (c), $\La_j \subset \DD_\de (\om_{0,j})$. Proposition \ref{p om path} and the fact that $\A$ is convex
imply that each $\La_j$ is path-connected. This proves \textbf{(1.i)}.

Taking $\om_j \in \La_j$ such that $\im \om_j = \max_{\om \in \La_j} \im \om$ and
$B_j \in \A$ such that $\om_j \in \Si (B_j)$, one can see that $(\om_j, B_j)$ is a local minimizer for $\Dr$. Since $\D$ is bounded, so are $\La_j$. Hence, the existence of local maximizers for $\Dr$ can be shown in the same way. This proves \textbf{(1.ii)}. The arguments for \textbf{(1.iii)} are similar.
After all these considerations,  statement \textbf{(2)} is obvious.
\end{proof}

\section{Locally extremal $\om$ are nonlinear eigenvalues}
\label{s Opt=>nl}

Recall that the set $\Si^\nl$ of nonlinear eigenvalues was introduced in Section \ref{s Snl pert}.
This section is devoted to the following theorem, which states that quasi-eigenvalues optimal in several senses lie in $\Si^\nl$ .

\begin{thm} \label{t nl ep}
A number $\om$ is a nonlinear eigenvalue if at least one of the following conditions hold:
\item[(i)] $\om$ is a quasi-eigenvalue of minimal decay for the frequency $\re \om$,
\item[(ii)] $\re \om \neq 0$ and $\om$ is a boundary point of $\Si [\A]$,
\item[(iii)] there exists $B_0$ such that $(\om,B_0)$ is a local extremizer for $\Dr$,
\item[(iv)] there exists $B_0$ such that $(\om,B_0)$ is a local extremizer of $\im \om$ for the frequency $\re \om$.
\end{thm}

The proof given below combines the results of Sections \ref{s Snl pert}-\ref{s 0fr} with the technique of local boundary points of $\Si [\A]$ introduced in \cite{Ka14}.

\subsection{Local boundary of $\Si [\A]$ and the proof of Theorem \ref{t nl ep}}
\label{ss proof nl}

\begin{defn}[cf. \cite{Ka14}] \label{d loc Bd}
Let  $B \in \A$ and $\om \in \Si (B)$. Then:

\item[(i)] the complex number $\om$ is called a \emph{$B$-local boundary point of the image} $\Si [\A]$ of $\A$ if $\om$
is a boundary point
of $\Si [\A \cap W]$ for a certain open neighborhood $W$ of $B$ in the norm topology of $L^1$.

\item[(ii)] $\om$ is called a \emph{$B$-local interior  point}  $\Si [\A]$ if $\om$ is an interior point
of the set $\Si [\A \cap W]$ for any open neighborhood $W$ of $B$ (in the norm topology of $L^1$).
\end{defn}

This definition is a particular case of \cite[Definition 4.1]{Ka14}. 
(In \cite{Ka14}, such $\om$ are called strongly local boundary points.)
Note that, if $B \in \A$, then a complex number $\om \in \Si (B)$ is a $B$-local interior  point of $\Si [\A]$ exactly when it is not a $B$-local boundary point of $\Si [\A]$.

The following statement is a particular case of \cite[Theorem 4.1 and Proposition 4.2]{Ka14}.

\begin{prop}[\cite{Ka14}] \label{p inter of image}
Let $B_0 \in \A$ and $\om \in \Si (B_0)$.
If $\cone \frac{\pa F (\om;B_0)}{\pa B}  [\A-B_0] = \CC$, then $\om $ is a $B_0$-local interior point of $ \Si  [\A]$.
\end{prop}

\textbf{Consider the case $\om \not \in \ii \RR$.} Combining Proposition \ref{p inter of image} and Theorem \ref{t char Si nl} with the equivalence (\ref{e cone <=> cone}), one easily gets the following result.

\begin{thm} \label{t nl ep loc}
If $\re \om \neq 0$ and $\om$ is a $B_0$-local boundary point of $\Si [\A]$ for a certain $B_0 \in \A$, then $\om \in \Si^\nl$.
\end{thm}

Now we are able to obtain Theorem \ref{t nl ep} in the case $\re \om = 0$. Indeed,
in this case, each of conditions (i)-(iv) implies that $\om$ is a $B_0$-local boundary point of $\Si [\A]$ for a certain $B_0$. Theorem \ref{t nl ep loc} complete the proof.

\textbf{The proof of Theorem \ref{t nl ep} for $\om \in \ii \RR$}
follows from Proposition \ref{p loc extr al=0}. Indeed, it is enough to prove that (iv) implies $\om \in \Si^\nl$ (see Remark \ref{r loc opt}).  Assume that $(\om,B_0)$ is a local extremizer of $\im \om$ for the frequency $0$. By Proposition  \ref{p loc extr al=0}, $B_0$ coincides either with $b_1$, or with $b_2$.

Suppose $B_0 = b_2$. Then $y(x) = e^{\ii \pi /4} \theta (x,\om;b_2) $ is a mode associated with the quasi-eigenvalue $\om$ of $b_2$. By (\ref{e theta real}) and Lemma \ref{arg th with alf=0} (iii), $\theta^2 >0$ everywhere on $[a_1,a_2]$ except possibly one point.  Hence $y^2 (x) \in \CC_+$ a.e., and so $b_2  = \left[ b_1  + (b_2 - b_1) \ChiCpl (y^2 ) \right]$ a.e. This means that $y$ is an eigenfunction of (\ref{e nonlin}), (\ref{e bc a-}), (\ref{e bc a+}) associated with the nonlinear eigenvalue $\om$.

If $B_0 = b_1$, we can put $y(x) = e^{-\ii \pi /4} \theta (x,\om;b_1) $.
This completes the proof.

\section{5-D reduction with the use of bang-bang equations}
\label{s finding}

\subsection{Finding of optimizers from the nonlinear spectrum}
\label{ss find o from nl}

Assume that we have found the set $\Si^\nl$ of nonlinear eigenvalues. Then the quasi-eigenvalues of optimal decay can be calculated with the use of the following formula:
\begin{equation} \label{e b min =}
\beta_{\min} (\alpha) = \min \{ -\im \om \ : \ \om \in \Si^\nl \ \text{ and } \ \re \om = \alpha \} .
\end{equation}
(the minimum exists for each $\alpha \in \re \Si [\A]$).
This statement follows from (\ref{e Snl subset}) and Theorem \ref{t nl ep}.
As a by-product we get 
\begin{equation} \label{e ReA=Re nl}
\re  \Si [\A] = \re \Si^\nl.
\end{equation}

Roughly speaking, \emph{the nonlinear eigenproblem (\ref{e nonlin}), (\ref{e bc a-}), (\ref{e bc a+}) excludes unknown
optimal $B_0$} (in any sense described above) from the optimization problem if corresponding optimal $\om_0$ is known. In this case, $B_0$ can be recovered by one of eigenfunctions of the nonlinear problem. Rigorously this is formulated in the following result.

Let us define a map $y \mapsto \B (y)$ from $C [a_1,a_2]$ to $L_{\RR}^1 (a_1, a_2)$ by
\[
\B (y) (\cdot)  = b_1 (\cdot)  + (b_2 (\cdot)  - b_1 (\cdot)) \ChiCpl (y^2 (\cdot)) .
\]

\begin{cor} \label{e nl B}
(a) Assume that statement (i) or statement (ii) of Theorem \ref{t nl ep} hold.
Then $\om$ is a quasi-eigenvalue of $B_0 \in \A$ if and only if  $B_0= \B (y) $  a.e. on $(a_1,a_2)$ for a certain eigenfunction $y$ of (\ref{e nonlin}), (\ref{e bc a-}), (\ref{e bc a+}).
\\[2mm]
(b) If at least one of the statements (iii), (iv) of Theorem \ref{t nl ep} hold,
then $B_0= \B (y) $  a.e. on $(a_1,a_2)$ for
a certain eigenfunction $y$ of (\ref{e nonlin}), (\ref{e bc a-}), (\ref{e bc a+}).
\\[2mm]
(c) If $\re \om \neq 0$ and $\om $ is a $B_0$-local boundary point of $\Si [\A]$ for a certain $B_0 \in \A$,
then $B_0= \B (y) $  a.e. on $(a_1,a_2)$ for
a certain eigenfunction $y$ of (\ref{e nonlin}), (\ref{e bc a-}), (\ref{e bc a+}).
\end{cor}

\begin{proof}
\textbf{(c)}  By Proposition \ref{p inter of image}, $\cone \frac{\pa F (\om_0;B_0)}{\pa B}  [\A-B_0] \neq \CC$.
Combining this and Theorem \ref{t char Si nl}, one obtains the desired statement.

In the case $\re \om \neq 0 $, \textbf{(b)} follows from \textbf{(c)}. In the case $\re \om = 0$,
\textbf{(b)} easily follows from the arguments at the end of Section \ref{ss proof nl}. Note that in the latter case, $B_0$ coincides either with $b_1$, or with $b_2$.

\textbf{(a)} The implication 'if' is obvious. Let us prove the part 'only if' in the case  $\re \om \neq 0 $. Assume that $\om$ is a quasi-eigenvalue of $B_0 \in \A$. Since $\om \in \Bd \Si [\A]$,  it is a $B_0$-local boundary point of $\Si [\A]$. Thus, (c) implies the desired statement.
In the case  when $\re \om = 0 $ and $\om$ is of minimal decay, the implication 'only if' follows from (b). 
\end{proof}

\subsection{Nonlinear eigenvalues that are not of minimal decay}

\begin{rem} \label{r B=b1b2}
If $B_0$ is optimal  in one of the senses of Corollary \ref{e nl B}, then $B_0$ is an extreme point of $\A$.
More precisely, Corollary \ref{e nl B} implies that
after a possible correction of $B_0$ on a set of measure zero,
$B_0 (x)$ satisfies at least one of the equalities $B_0 (x) = b_1 (x)$, $B_0 (x) = b_2 (x)$
at every $x \in (a_1,a_2)$. 
If, additionally, $b_1$ and $b_2$ are piecewise constant with a finite number 
of intervals of constancy (see \cite{Ka13} for the precise formulation in 1-side open case), 
then $B_0$ is so.  
The last  statement follows from Lemma \ref{l turn int} (v).
(For optical problems this means that $B_0$ represents a multilayer resonator 
with a finite number of layers.)
\end{rem}

\begin{prop} \label{p int points}
Let $\re \om \neq 0$ and $\om \in \Si (B_0)$ for a certain $B_0 \in \A$ that is not an extreme point of $\A$. Then $\om$ is an interior point of $\Si [\A]$. 
\end{prop}

The proposition follows from the combination of Proposition \ref{p inter of image} 
with Theorem \ref{t char Si nl} and its proof.

\begin{rem} \label{r interior}
Applying Proposition \ref{p int points} in the settings of Theorems \ref{t small diel contrast} and 
\ref{t sm contr gen}, one can see that if $\om_{0,j} \not \in \ii \RR$, than the corresponding path connected component $\La_j$ (of the set $\Si [\A]$) has a nonempty interior.
\end{rem}

\begin{ex} \label{ex interior}
Let $B_0 \equiv b \in \RR$, $b \neq \nu_{1,2}^2$, and let $\om_{0,j}$ be equal to $\om_j (b)$ of (\ref{e om n hom}) for $1 \le j \le n$.  Let $\ep$ and $b_{1,2}$ satisfy statement (1) of Theorem \ref{t sm contr gen}. Then \emph{there exist $\om \in \Si^\nl \setminus \ii \RR$ that is not a quasi-eigenvalue of minimal decay.}

Indeed, by Remark \ref{r interior}, the interiors $\Int \La_j$ of $\La_j$ are not empty. So for any $\alpha$ from their real projection $ \re [\Int \La_j]$, Theorem  \ref{t sm contr gen} and its proof imply the existence of a local minimizer $(\om,B)$ of $\im \om$ such that $\om$ is not of minimal decay for $\alpha$
(see Fig.~1 in Section \ref{s opt problem})
\end{ex}

\subsection{Nonlinear eigenvalues as zeros of a function of three variables}
\label{ss Sinl=0}

\begin{thm} \label{t uni B}
Assume that $b_{1,2}$ satisfy (\ref{e b12 cond}) and
\begin{equation} \label{e a b1>0}
b_1 (x) > 0 \  \text{ a.e. on the set $E$ (recall that $E= \{ x \ : \ b_1 (x) \neq b_2 (x) \}$).} 
\end{equation}
Let $a \in [a_1,a_2]$ and $c_0, c_1 \in \CC$. Then there exists a unique solution $y$ to the problem
\begin{equation} \label{e ini cond}
y (a) = c_0 , \qquad y'(a) = c_1 
\end{equation}
for the nonlinear equation 
\begin{equation} \label{e nonlin z}
y'' (x)  = - z^2 y (x) \left[ b_1 (x) + [b_2(x) - b_1(x)] \ChiCpl (y^2 (x) ) \right] .
\end{equation}
\end{thm}

The proof includes a technical part that considers a number of cases for placing of $c_{0,1}$ in $\CC$. It is postponed to Appendix \ref{ss uni nl}.

\begin{ex}[Nonuniqueness] If assumption (\ref{e a b1>0}) is dropped, then the statement of Theorem \ref{t uni B} is incorrect. To see this, consider the case when $b_{1,2}$ are constants and $0= b_1< b_2$. Let $z^2 \in \CC_-$, $a=a_1$, $c_0>0$, and $c_1 = 0$. Then, for small enough positive $\ep$, the functions 
$y_\ep$ defined by $y_\ep (x) = c_0$ for $x \in [a_1,a_2 - \ep]$ and by $y_\ep (x) = 
c_0 \cos (z b_2 [x-a_2+\ep])$ for $x \in [a_2-\ep,a_2]$ are solutions to (\ref{e ini cond}), (\ref{e nonlin z}) on $[a_1,a_2]$.
\end{ex}

We assume (\ref{e a b1>0}) in the rest of this section. Assuming also $\xi \in (0,\pi]$, let us define the function $\Theta (x;\xi,z)$ as the solution to 
(\ref{e nonlin z})  satisfying the initial conditions 
\begin{equation}
\Theta (a_1;\xi,z) = e^{\ii \xi},  \ \ \  \pa_x \Theta (a_1;\xi,z) = -\ii z \nu_1 e^{\ii \xi} . \label{e Th}
\end{equation}
Note that $\Theta$ satisfies (\ref{e bc a-})  for any $\xi$ (with $\om =z$).
Plugging $\Theta$ into the second boundary condition (\ref{e bc a+}), let us introduce the complex-valued function  
\[
F_\nl (\xi,z) := \Theta (a_2;\xi,z) + \frac{\ii \pa_x \Theta (a_2;\xi,z)}{z\nu_2}   \qquad \text{ for } z \in \CC \setminus {0}, 
\]
and $F_\nl (\xi,0) := 1 + \frac{\nu_1}{\nu_2}$ (the logic of this definition is the same as in Section \ref{ss F}).

The following corollary essentially reduce the problem of finding of $\Si^\nl$ (and so, of finding the function $\beta_{\min}$) to a  finite-dimensional question of finding of zeros for a function of three real variables.  
Consider the fiber of the function $F_\nl$ over $0$,
\[
F_\nl^{-1} (0) := \{(\xi,z) \in (0,\pi] \times \CC \ : \ F_\nl (\xi,z)  = 0\} ,
\]
and the projection of $F_\nl^{-1} (0)$ to the $z$-plane
\[
\pr_z F_\nl^{-1} (0)  := \{ z \in \CC \ : \ F_\nl (\xi,z) = 0 \text{ for certain } \xi \in (0,\pi] \} .
\]

\begin{cor} \label{c Si nl =}
Suppose (\ref{e a b1>0}). Then
$\Si^\nl = \pr_z F_\nl^{-1} (0)$.
\end{cor}

\begin{proof} The inclusion $\supset$ is obvious. Let us prove $\subset$.
Assume that $\om \in \Si^\nl$ and $y$ is an associated eigenfunction of (\ref{e nonlin}), 
(\ref{e bc a-}), (\ref{e bc a+}). 
Then (\ref{e bc a-}) and Theorem  \ref{t uni B} imply $y(a_1) \neq 0$. Consider two functions $y_\pm (x):= \pm y(x) / |y(a_1)|$. They are also eigenfunctions of (\ref{e nonlin}), (\ref{e bc a-}), (\ref{e bc a+}). (Note that $\B (y) = \B (y_-) = \B (y_+)$). 
For one of the functions $y_\pm$ (let us denote it by $y_1$), the equality  
$y_1 (a_1) = e^{\ii \xi}$ holds with certain $\xi \in [0,\pi)$. By Theorem  \ref{t uni B},
this function coincide with $\Theta (\cdot;\xi,\om)$. So $\Theta (\cdot;\xi,\om)$ is an 
eigenfunction of (\ref{e nonlin}), (\ref{e bc a-}), (\ref{e bc a+}). Thus, (\ref{e bc a+})
yields $F_\nl (\xi,\om) = 0$. 
\end{proof}

For the particular case considered in Example \ref{ex 1-side D} and $\om \not \in \ii \RR$, Corollary \ref{c Si nl =} was announced without proof in \cite{Ka13_Opt}.

\section{A numerical experiment for constant side constraints}
\label{s NumEx} 

To check applicability of the method of Section \ref{ss Sinl=0} to calculation of quasi-eigenvalues of minimal decay, we take $\nu_1 = 1$, $\nu_2 =+\infty$, $a_1=-1$, $a_2 = 0$, and the constant functions $b_1 = 90$, $b_2 = 110$.
% and consider the region  $\D = \{\re z \in (0, 1.2) , \ \im z \in (-0.015, 0) \}$.
Then $F_\nl (\xi,z) = \Theta (0; \xi, z)$ and 
$\Si^\nl = \{ z \in \CC_- \ : \ F_\nl (\xi,z) = 0 \text{ for certain } \xi \in (0,\pi] \}$.

To find $\Si^\nl$ numerically, we take small $\ep$ and approximate the sub-level set 
$
\Z_\nl (\ep) :=  \{ z \in \CC \ : \ \inf_{\xi \in (0,\pi] } |F_\nl (\xi )| \le \ep \}
$
by its discrete versions  
$\Z (h_R,h_I,h_{\arg}; \ep)$ consisting of numbers $z = z_1 + \ii z_2$ with $z_1 \in h_R \ZZ$, $z_2 \in h_I \ZZ$, such that 
$\max_{1 \le n \le N} |F_\nl (n h_{\arg}, z)| \le \ep$. The angular step $h_{\arg}$ is connected with $N$ by $h_{\arg} = \pi / N$.

The value of $F_\nl (\xi,z) = \Theta (0; \xi, z)$ 
is computed by the shooting method
applied to the solution $\Theta$ of the equation 
\begin{equation} \label{e nleq const}
y'' (x)  = - z^2 y(x)  \left[ b_1 + [b_2 - b_1] \ChiCpl (y^2 (x) ) \right] .
\end{equation}
This does not require finite-difference approximation. 
Indeed, for $z^2 \in \CC_-$, the equation (\ref{e nleq const}) has constant 
coefficients on each interval where $\Theta^2$ stays in 
one of half-planes $\CC_\pm$ and admit an explicit analytic solution depending only on the initial 
values in the left end of this interval. 
These values and the lengths of intervals of constancy 
can be computed iteratively starting from initial conditions (\ref{e Th}).

The result of such a computation of $\Z_\nl (\ep)$ with the level $\ep = 5 \times 10^{-5}$  in the domain $\D = \{\re z \in (0, 1.2) , \ \im z \in (-0.015, 0) \}$ are plotted in Figure \ref{f Snl} (a). We take $N = 360$. The other step parameters are changed through $\D$ from 
$(h_R, h_I) = (2.4 \times 10^{-4}, 5.5 \times 10^{-5})$ to 
$(h_R, h_I) = (2.7 \times 10^{-3}, 2.3 \times 10^{-4})$ to 
provide better resolution for the regions where $\Z_\nl (\ep)$ is 
concentrated and to reduce computation time in the regions where 
the values of function $\M (z) := \max_{1 \le n \le N} |F_\nl (n h_{\arg}, z)| $ 
are essentially greater than $\ep$.

Comparing the shape of this approximation of $\Si^\nl \cap \D$ with 
the fact that $\Bd \Si [\A] \setminus \ii \RR \subset \Si^\nl$ (see Theorem \ref{t nl ep} (ii)) and Theorems \ref{t small diel contrast}-\ref{t sm contr gen}, we conclude that the considered example fits into the description of the small contrast case at least in the domain $\D$, and that $\Z_\nl (\ep)$ provides an approximate shape not only for $\Si^\nl \cap \D$, but also for $\Bd \Si [\A] \cap \D$.
%It is clear that, in this experiment, either 
%$\Si^\nl \cap \D = \Bd \Si [\A] \cap \D$, or, at least, $\Si^\nl \cap \D$ 
%and $\Bd \Si [\A] \cap \D$ are 'close' to each other on the scale of our 
%choice of the resolution parameters $h_R$, $h_I$. 
The form of  $\Bd \Si [\A] \cap \D$ together with Theorem \ref{t sm contr gen} makes it possible to 
find approximately the part $\Si [\A] \cap \D$ of the set of admissible resonances, which consists of 
four quasi-eigenvalue 'clouds' $\La_1$, \dots, $\La_4$ (see Remark \ref{r drawing}).

\begin{figure}[h]
% \centering

\begin{tabbing}
\hspace{0.58\linewidth} \= \hspace{0.38 \linewidth} \kill
 \footnotesize  (a)   \> \footnotesize (b) \\
 \begin{minipage}[t]{0.58\linewidth}  
   \includegraphics[width=0.99\linewidth]{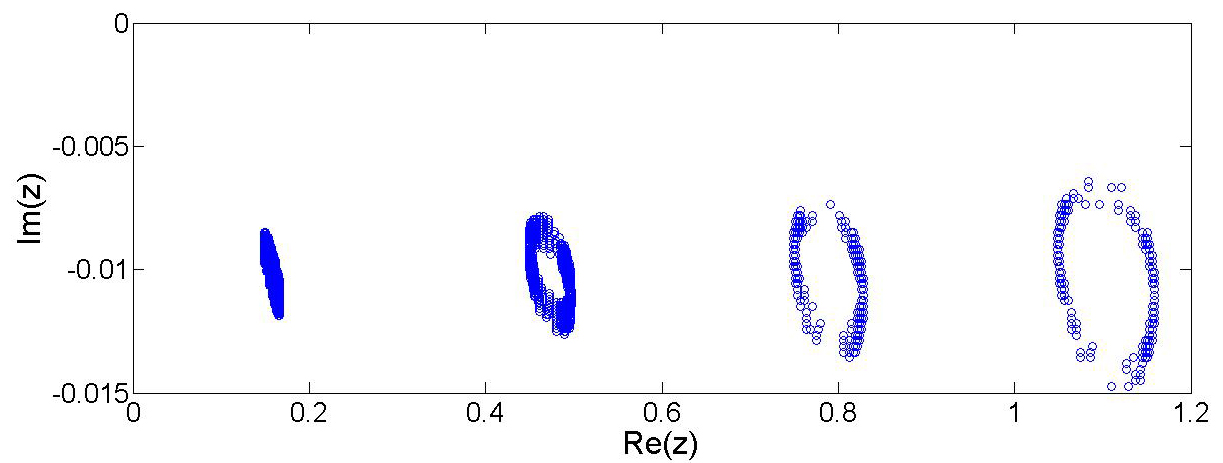}
\end{minipage}
\> 
 \begin{minipage}[t]{0.38\linewidth} 
\includegraphics[width=0.99\linewidth]{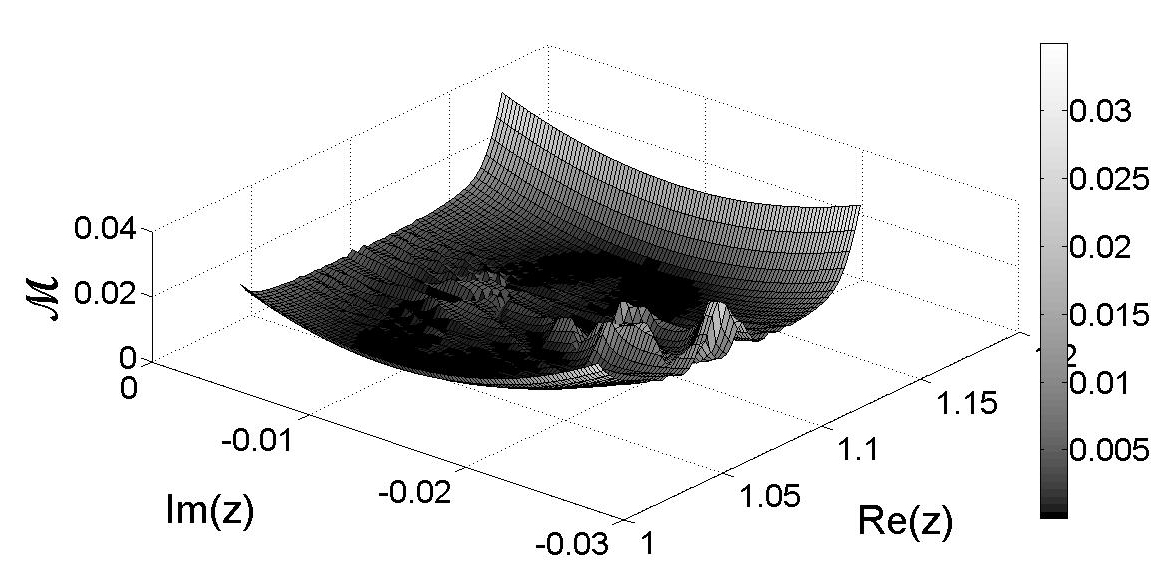} 
\end{minipage} \\
\end{tabbing}

    \caption{\footnotesize MATLAB computation of: (a) $\Si^\nl$ in the rectangle $\D$, 
    (b)  the graph of  $\M (z) := \max_{1 \le n \le N} |F_\nl (n h_{\arg}, z)| $ in the  neighborhood 
$\D_4 = \{\re z \in (1, 1.2) , \ \im z \in (-0.025, 0) \} $ of $\La_4$. }
    \label{f Snl}
\end{figure}

While the shape of $\Z_\nl (\ep)$ in the rectangular neighborhood 
$\D_4 = \{\re z \in (1, 1.2) , \ \im z \in (-0.025, 0) \} $ 
still resembles a boundary of a closed bounded path-connected component $\La_4$
of $\Si [\A]$ (and even suggests that $\La_4$ is simple connected), the graph of 
function $\M (z)  $   
exhibits a close to periodic
pattern of ridges and canyon in the regions with $\im z$ close to 
$-0.007$ and $-0.013$, see Figure \ref{f Snl} (b). 
For the study of $\La_4$,
the  value of $h_{\arg} = \pi/360$ is too large. When $h_{\arg}$ is decreased to $\frac{\pi}{360} \times 10^{-3}$ while keeping $h_R \ge 10^{-3}$ and $h_I \ge 10^{-5}$, only one sharp canyon that 
corresponds to $\Bd \Si [\A] \cap \D_4$ remains.

The values of the minimal decay rate function $\beta_{\min}$ for $\alpha$ in 
the intervals $\re \La_1$, \dots , $ \re \La_4$ can be computed   
using (\ref{e b min =}).  We have done these computations for some of the values of $\alpha$,  see Table 1.
The step $h_I$ is $5\times 10^{-7}$ for the first and second values of $\alpha$ and $h_I =5\times 10^{-6}$ for the others, $h_{\arg}$ varies from $\frac{\pi}{180}\times 10^{-2}$ to $\frac{\pi}{360}\times 10^{-3}$.

\begin{figure}[h]
 \label{f Table1}
 % \centering
 
\begin{tabular}{|l|l|l|l|l|l|l|}
\hline
$\alpha$                          &  \; 0.14977  \;& \; 0.14983 \;& \; 0.16557  \;&\; 0.44931   \;    &\;  0.46050 \;   &  \; 0.49674 \;\\ \hline
$ \beta_{\min} (\alpha) $& \; 0.009119  \;& \; 0.009120 \; &\; 0.01115   \;& \; 0.00910 \;  & \; 0.00846 \; & \; 0.01115 \; \\  \hline
\end{tabular}
\\[1em]
\begin{tabular}{|l|l|l|l|l|l|l|}
\hline
$\alpha$                          &  \; 0.74884  \; \;& \; 0.77200 \; \; & \; 0.82790  \;&\; 1.04838   \;    &\;  1.08800 \;   &  \; 1.15905 \;\\ \hline
$ \beta_{\min} (\alpha) $&  \; 0.00910  \; \;& \; 0.00766 \; \; &\; 0.01113   \;& \; 0.00909 \;  & \; 0.00689 \; & \; 0.01109 \; \\  \hline
\end{tabular}\\[1mm]
Table 1: Some of values of the minimal decay rate function $\beta_{\min} (\alpha)$.
%    \caption{}      
\end{figure}

For $\alpha=1.088$, we have found the structure $B$ generating the quasi-eigenvalue of minimal decay 
$\alpha- \ii \beta (\alpha)$. It consists of $7$ layers. With $x_0 = -1$ and $x_7=0$  for 
the endpoints of the interval $[a_1,a_2]= [-1,0]$, the computed  intervals of constancy of $B$ 
can be written in terms of the boundaries of the layers (switch points) $x_j$, $j=1,\dots, 6$. 
Namely, $B(x) = 110$ for $x \in (x_{2k},x_{2k+1})$, $k=0,\dots,3$, and $B(x)=90$ for $x \in (x_{2k+1},x_{2k+2})$ for $k=0,\dots,2$.
Approximate values $\wt x_j$ for $x_j$ are given  in Table 2 together with absolute error 
estimates $\ep_j$.

\begin{figure}[h]
 \label{f Table 2}
 % \centering
  
\begin{tabular}{|c|c|c|c|c|c|c|}
\hline
  $j$ &   1 & 2 & 3  & 4      &  5   &   6 \\ \hline
  $\wt x_j $ & $-0.86237855$ &     $-0.71669445$ &     $-0.57371202$ &          $ -0.43130025$ &      $-0.28563822$ &      $ -0.15697427$   \\  \hline
  $\ep_j $ & $\ \ 0.00000028$ & $\ \ 0.00004800$ & $\ \ 0.00000900   $  & $ \ \  0.00007000 $  & $\ \ 0.00002000 $ & $ \ \ 0.00011200  $ \\  \hline
\end{tabular}\\[1mm]
{\small Table 2: Switch points and their absolute error estimates for $B(x)$ of minimal decay for $\alpha=1.088$.}
%    \caption{}      
\end{figure}

One of possible improvements of the proposed method may be concerned with calculation 
of zeros of the function $F_\nl (\xi,z)$. 
Indeed, we used the simplest straightforward approach based on sub-level sets of the function $\M$. 
More elaborated approaches that do not neglects the information contained in $\arg F_\nl (n h_{\arg}, z)$ should be more efficient. 

\section{Discussion}
\label{s dis}

Earlier Engineering and Numerical Optimization papers were concerned mainly with 
the case when the constraint functions $b_1$ and $b_2$ are constants.
In this section, we (sometimes speculatively) compare our results for this case 
with earlier suggestions for optimal designs and with discussions of the proper statement of the problem. 

\subsection{Alternating almost periodic structures with a centered defect.} 
\label{ss period defect}

Contemporary designs of high-Q optical cavities usually involve incorporation of defects 
into a periodic structure composed of two materials \cite{AASN03,NKT08,KTTKRN10}. 
For cavities with 1-D geometry \cite{KS08,NKT08,KTTKRN10,BPSchZsch11}, the base periodic 
structure consists of alternating layers (say with widths $\ell_1$ and $\ell_2$ ) of two materials with 
different permittivities (say $b_1$ and $b_2$, respectively). The Physics justification for such designs is 
that the periodic structure forms a distributed Bragg reflector  with high reflectivity for certain bands of frequencies (stopbands). 
When a defect is introduced, the waves with such frequencies are 
expected to accumulate and to be well confined in the region surrounding the defect \cite[Chapter 4]{JJWM08}, \cite{MHvG08}.    
If the fabrication process involves material with 
permittivities from $b_1$ to $b_2$, Remark \ref{r B=b1b2} 
 implies that 
the best confinement has to be produced by structures $B$  that consist only of layers with $B(x)=b_1$ and $B(x)=b_2$
(for one-side open case, see \cite{Ka13,Ka13_Opt}).
In other words, both the base structure and the defect have to consist of the two materials with extreme allowed permittivities.

Basing on various experimental and computational approaches, two types of designs with defects were proposed: 
\begin{itemize}
\item[(a)] a localized defect in a center of a periodic structure, e.g., \cite{AASN03,KS08,DSSW13},
\item[(b)] the  size-modulated 1-D stack
 \cite{NKT08,KTTKRN10,BPSchZsch11}.
\end{itemize}  
Designs of type (b) assumes gradually changing deviations from periodicity in the widths of layers.  

For the structure of minimal decay computed in Section \ref{s NumEx}  (see Table 2), one can observe that it is \emph{close to alternating periodic with a period consisting of one layer with $B(x)=b_1$ and the other layer with $B(x)=b_2$} (since the structure is finite,  either the first, or the last layer has to be treated separately). The widths of layers $[x_1,x_2]$, $[x_3,x_4]$, $[x_5,x_6]$ with $B=b_1$ \emph{gradually decrease from 
$0.146$ to
$0.129$}, and 
the widths of layers $[x_{2k},x_{2k+1}]$, $k=0,\dots,3$, corresponding to $B=b_2$ \emph{gradually increase from 
$0.138$ to  
$0.157$}. This resembles designs of type (b).

Let us note that the 1-D structures considered in \cite{KS08,NKT08,KTTKRN10,BPSchZsch11} are symmetric. If we consider the numerical experiment of Section \ref{s NumEx} in the context of optimization of 
odd modes of a symmetric structure in the interval $(-1,1)$ (cf. Example \ref{ex sym}), then the structure 
$B$ has to be extended to $(0,1)$ by $\wt B(\pm x) = B( x)$, $x \in (-1,0)$. 
Then the layer $[x_6,-x_6]$ is approximately twice larger of the nearby layers
and so the structure does not resemble the size-modulated 1-D stack designs of
 \cite{NKT08,KTTKRN10,BPSchZsch11}.
 
 The symmetric structure $\wt B$ can be considered as a \emph{combination of designs (a) and (b).} It has \emph{a localized defect consisting of the central layer $[x_6,-x_6]$, and this defect is surrounded by alternating layers with gradually changing widths.}

In  \cite{NKT08,KTTKRN10,BPSchZsch11}, the quadratic deviation laws were employed to vary the widths of layers. From the point of view of analytic optimization, it would be interesting to find dependence of the base periodic structure, the parameters of the defect, and width deviations on the frequency of corresponding optimal resonance.

\subsection{Global minimization of decay rate.}
\label{ss GlobMin}

Well-posedness of global maximization for $Q$-factor related functionals 
(without frequency restrictions) was discussed in \cite{KS08,LJ13}. 
For 1-D resonators with constant side constraints and fixed length $a_2-a_1$, 
the global maximizer for $Q$-factor does not exist. 
Indeed, if we take a homogeneous structure $B(x) = b $ 
(with constant $b \in [b_1,b_2]$, $b \neq \nu_{1,2}^2$), 
then (\ref{e om n hom}) implies that $Q(\om_n) = - \frac{\re \om_n}{2 \im \om_n} $ go to $+\infty$ 
as $|n| \to \infty$.

Below, basing our intuition on the numerical experiment of 
Section \ref{s NumEx}, we discuss the question of global minimization of the decay rate 
$|\im \om|$ (cf. \cite{HBKW08}).

It can be seen from Figure \ref{f Snl} (a) that quasi-eigenvalue `clouds` corresponding to higher frequencies are `wider'. It is natural to 
expect that for $|\im \om|$ great enough they 
intersect each other and form one large unbounded cloud. This conjecture is supported also by our study of clouds' projections (see Section \ref{ss ex A} and Appendix \ref{a adm fr})
and by the results on existence of multiple quasi-eigenvalues (see Section \ref{ss mult}). In particular, we already know that the minimal decay function is defined for large enough $|\alpha|$
(Proposition \ref{p adm fr} (iii)).

It also can be seen from Figure \ref{f Snl} (a) and Table 1 that each of subsequent clouds $\La_1$, \dots, $\La_4$, comes closer to the real line than previous. (The values  
$\beta (0.14977) \approx 0.009119$, $\beta (0.4605) \approx 0.00846$, $\beta (0.772) \approx 0.00766$, and $\beta (1.088) \approx 0.00689$ approximately corresponds to minima of decay rate for the clouds $\La_1$, \dots, $\La_4$.) Basing on the above observations, we propose for the case when $b_{1,2}$ are constants and $0 \le b_1 <b_2 < +\infty$ the following

\begin{conjecture}
 The inequality $\liminf_{|\alpha| \to \infty} \beta (\alpha) <  \beta (\alpha_0) $ 
 holds for all admissible frequencies $\alpha_0$. (In particular, the global infimum of the decay rate $\inf_{\om \in \Si [\A]} |\im \om|$ is not achieved over $\A$.)
\end{conjecture}

A stronger form of this conjecture that assumes  existence of  
$\lim_{|\alpha| \to \infty} \beta (\alpha)$ could help to explain the sliding effects in gradient ascent numerical experiments of \cite{HBKW08,KS08}. Namely, in these experiments, iteratively improved resonances $\om^{[n]}$ were approaching to the real axis with simultaneous growth of frequency 
$\re \om^{[n]}$. This could happen if $\om^{[n_0]}$ for some iteration $n_0$ reach (or come very close to) the Pareto optimal frontier 
$\{ \alpha - \ii \beta (\alpha) \ : \ \alpha \in \re \Si [\A] \}$ and then move along this frontier 
to $\infty$ with growing $\re \om$ and decreasing $|\im \om|$. In our opinion, the comparison 
of results of 
simulations \cite{BPSchZsch11,HBKW08,KS08,NKT08} 
suggests that $\lim_{|\alpha| \to \infty} \beta (\alpha) = 0$.

It is also interesting to consider the value of 
$\liminf_{|\alpha| \to \infty} \beta (\alpha) $ from the point of view of the 
question of uniform separation of the set of resonances $\Si (B)$ 
from the real line (the question of exponential energy decay in the sense of \cite{CZ94,CZ95,CO96}). 
The absence of quasi-eigenvalues in the certain strip $\{-\beta_0 < \im z  < 0 \}$ is presently 
known only under additional  conditions involving bounds on total variation of $B$ \cite{CZ95}. 
So it is difficult to expect that, for certain $\beta_0>0$, the set 
$\Si [\A] $ does not intersects with the strip $\{-\beta_0 < \im z  < 0 \}$ in the case of 
the admissible family $\A$ restricted only by constant side constraints.

\appendix

\section{Appendix}
\label{s Ap}

\subsection{Analyticity, the proof of Lemma \ref{l an}}
\label{ss an and exist}

\textbf{(i)} follows from (ii) and (\ref{e F}). \textbf{Let us prove
(ii)} for $(z, B) \mapsto \psi (\cdot ,z; B)$. Consider the Maclaurin series (in $z\in \CC$ and $B \in L_\CC^1$)
\begin{eqnarray}
\Psi (x , z ; B) = (x-a_1) - \psi_1 (x; B) z^2 + \psi_2 (x; B) z^4 - \psi_3 (x; B)  z^6 + \dots , \label{e pow ser}\\
\psi_0 (x; B) \equiv x-a_1 , \ \  \vphi_{j} (x; B) = \int_{a_1}^x
(x - s) \, \vphi_{j-1} (s; B) \, B (s)\dd s , \ \ j \in \NN ,
\label{e psi j}
\end{eqnarray}
as a $C [a_1,a_2]$-valued series. One can see that
\begin{equation} \label{e psi a2}
\text{$|\psi_{j} (x;B)| \le \psi_j (x; |B|)  \le \psi_j (a_2; |B|)$ \ \ for all $x \in [a_1,a_2]$.}
\end{equation}
It can be shown in the way similar to \cite[Sect. 2]{KK68_II} that
\begin{equation} \label{e psi j<}
\psi_{j} (a_2; |B|) \leq \frac{2^j (a_2-a_1)^{j+1} \| B \|_1^j}{(2j) !} , \ \ \ j \in \NN  .
\end{equation}
Hence, the series (\ref{e pow ser}) converge uniformly on every bounded set of
$\CC \times L_\CC^p $. So  (\ref{e pow ser}) defines an analytic map
from $\CC \times L_\CC^p $ to $C [0,\ell]$. Since (\ref{e pow ser})
satisfies $\pa_x^2 y = - z^2 B y $, $y(a_1)=0$, and $\pa_x y(a_1) = 1$, we see that
$\Psi (x , z ; B) = \psi (x , z ; B)$.
The equality $\pa_x^2 \psi = - z^2 B \psi $ yields that
$\pa_x^2 \psi \in L_\CC^p $ and, moreover, that $(z, B) \mapsto \pa_x^2 \psi (\cdot ,z; B)$ is an analytic map
from $\CC \times L_\CC^p$ to $ L_\CC^p$. Hence, (\ref{e pow ser}) defines an analytic map
from $\CC \times  L_\CC^p$ to $W^{2,p}$.
This map is also bounded-to-bounded due to (\ref{e psi a2}), (\ref{e psi j<}) and the equalities $\pa_x^2 \psi = - z^2 B \psi $, $\psi (a_1)=1$ .

Similarly one can prove that $(z, B) \mapsto\phi (\cdot ,z; B)$ is an analytic and bounded-to-bounded map
from $\CC \times  L_\CC^p$ to $W^{2,p}$ (cf. \cite{KK68_II,Ka14}).
Finally, since $\theta  =  \vphi - \ii z \nu_1 \psi $, we obtain the same statement for the map
$(z, B) \mapsto \theta (\cdot ,z; B)$. This completes the proof.

\subsection{Directional derivatives, the proof of Lemma \ref{l der F}}
\label{ss proof of der F}

\textbf{Differentiation w.r.t. $B$.}
For any $z \in \CC$, the functions
$\theta (x,z;B)$ and $\pa_x \theta (x,z;B)$ satisfy (\ref{e int ep th}) and
\begin{eqnarray}
%\theta(x,z;B) & = & 1 - z^2 \int_{a_1}^x (x-s) \ B (s) \ \theta (s,z;B) \ \dd s -iz\nu_1(x-a_1), \label{e phi z int}\\
\pa_x \theta (x,z;B) & = & -iz\nu_1-z^2 \int_{a_1}^x \ B(s) \ \theta (s,z;B) \ \dd s . \label{e pax phi z int}
\end{eqnarray}
To find directional derivatives $\frac{\pa \theta}{\pa B} (V)$ and $\frac{\pa \ \pa_x \theta}{\pa B} (V)$,
we differentiate  these equalities by definition using Lemma \ref{l an}.
We get
\begin{gather}
 \frac{\pa \theta (x,z;B)}{\pa B} (V)  = -
z^2 \int_{a_1}^x (x-s) B (s) \frac{\pa \theta (s,z;B)}{\pa B} (V) \dd s %\notag \\ & & \qquad \qquad
- z^2 \int_{a_1}^x (x-s) V (s)  \theta (s,z;B) \dd s,
  \label{e int paBD ph} \\
\frac{\pa \  \pa_x \theta (x,z;B)}{\pa B} (V)  =  -z^2 \int_{a_1}^x B(s)  \frac{\pa \theta (s,z;B)}{\pa B} (V) \dd s -z^2 \int_{a_1}^x V (s) \theta (s,z;B) \dd s.
 \label{e int paBD pax ph}
\end{gather}
It follows from (\ref{e int paBD ph}) that $\frac{\pa \theta (x,z;B)}{\pa B} (V)$
is the solution $y(x)$ of the initial value problem
\begin{equation} \label{e nh bvp}
 y''(x) + z^2 B(x) y(x) = f(x) ,  \ \ y(a_1) = 0, \ y'(a_1) = 0
\end{equation}
with $ f(x) = - z^2 V (x) \theta (x,z;B)$.
Further, (\ref{e int paBD pax ph}) can be rewritten as
$ \frac{\pa \  \pa_x \theta (x,z;B)}{\pa B} (V)= y'(x)$
(as a by-product, we get
$ \frac{\pa \  \pa_x \theta (x,z;B)}{\pa B} (V)  = \pa_x \left[ \frac{\pa \theta (x,z;B)}{\pa B} (V) \right]$).
Solving (\ref{e nh bvp}) by variation of parameters, one can find $y$,$y'$, and, in turn, $y(x) + \frac{\ii y'(x)}{z\nu_2}$.
For $z \neq 0$, $ \bigl[y(x) + \frac{\ii y'(x)}{z\nu_2}\bigr]_{x=a_2}$ equals
\begin{multline*}
- \left[ \theta(a_2) +\frac{i}{z\nu_2} \pa_x \theta(a_2)\right]
\int_{a_1}^{a_2} f(s)\  \psi (s,z;B)\ \dd s+ \left[ \psi(a_2) +\frac{i}{z\nu_2} \pa_x \psi(a_2) \right]
\int_{a_1}^{a_2} f(s)\  \theta (s,z;B)\ \dd s  .
\end{multline*}

Substituting $f$, and $z=\om \in \Sigma (B)$ , we get
\begin{eqnarray*}
\theta (a_2,\om;B) & + & \frac{\ii}{\om\nu_2} \pa_x \theta (a_2,\om;B)  =  0 \label{e th bc a2}\\
\text{and} \qquad
 \frac{\pa F (\om,B)}{\pa B} (V) & =&  -\om^2 \left[ \psi(a_2) +\frac{i}{z\nu_2} \pa_x \psi(a_2) \right] \int_{a_1}^{a_2}
\theta^2 (s,\om;B) \ V (s) \ \dd s . \notag
\end{eqnarray*}
To obtain (\ref{e paBD F}), it remains to note that
\begin{equation*} \label{e 1 / pa x th}
 \psi(a_2) +\frac{i}{z\nu_2} \pa_x \psi(a_2) = - \frac{1}{\pa_x \theta(a_2)}.
\end{equation*}
The last identity easily follows from
the constancy of the Wronskian
%\begin{equation} \label{e Wrons}
$\left| \begin{array}{cc}
\theta & \psi \\
\pa_x \theta & \pa_x \psi
\end{array} \right| \equiv 1 $.
%\end{equation}

\textbf{Differentiating  (\ref{e int ep th}) and  (\ref{e pax phi z int}) w.r.t. }$z$,
we see that $\pa_z \theta $
is given by the solution $y$ of the problem
$$y''(x) + z^2 B(x) y(x) = f(x) ,  \ \ y(a_1) = 0, \ y'(a_1) = -i\nu_1$$
with $ f(x) = - 2 z B (x) \theta (x,z;B)$ and that $\pa_z \pa_x \theta = \pa_x \pa_z \theta$. Hence, for $z=\om \in \Sigma (B)$,
\begin{multline*}
\pa_z F (\om;B)  =  \pa_z \theta (a_2,\om;B) +\frac{\ii \pa_z \pa_x \theta (a_2,\om;B)}{\om\nu_2} - \frac{\ii \pa_x \theta (a_2,\om;B)}{\om^2\nu_2} 
= \\ = 
y(a_2) +\frac{\ii y' (a_2)}{\om\nu_2} - \frac{\ii \pa_x \theta (a_2,\om;B)}{\om^2\nu_2}.
\end{multline*}
Solving for $y$ by variation of parameters as before, one can easily get (\ref{e paz F}).
% with the use of (\ref{e 1 / pa x th}) and (\ref{e th bc a2}).

\subsection{Admissible high frequencies, the proof of Proposition \ref{p adm fr} (iii)}
\label{a adm fr}

It is enough to prove the statement of Proposition \ref{p adm fr} (iii) under the assumption that $b_{1,2}$ are constants satisfying $b_1<b_2$. We consider in details several essentially different cases of arrangement of the intervals $[\nu_1,\nu_2]$ and $[b_1^{1/2},b_2^{1/2}]$, and very briefly the other cases (for which the proof can be obtained by simple modification of arguments). 

\emph{Step 1. For the cases when $0<b_1^{1/2} <b_2^{1/2}<\nu_1$ or $\nu_2 \le b_1^{1/2}$,}
let us show that 
\begin{equation} \label{e re Si A out}
\re \Si [\A] \supset \left\{ | \alpha | \ge 
\frac{\pi }{(a_2-a_1) b_2^{1/2} } 
\left\lceil  \frac{ b_1^{1/2} }{ b_2^{1/2} -b_1^{1/2} }   \right\rceil
\right\} .
\end{equation}
Indeed, for each $n \in \NN$ consider the intervals $\Sset_n$ covered by $\re \om_n (b)$
when a constant $b$ changes in the range $(b_1,b_2]$. By (\ref{e om n hom}),
$\Sset_n = \frac{\pi n}{a_2-a_1} [b_2^{-1/2},b_1^{-1/2})$. 
If the right end of $\Sset_n$ is $\ge$ than the left end of $\Sset_{n+1}$ for all $n \ge n_0 \in \NN$, then $\re \Si [\A]$ contains the set 
$
\cup_{n\ge n_0} \Sset_n = \left[ \frac{\pi n_0}{(a_2-a_1)b_2^{1/2}} , +\infty \right)
$
and also its reflection w.r.t. $0$, see Lemma \ref{l prop Si} (ii).
Estimating $n_0$ from the inequality $b_1^{-1/2} n \ge b_2^{-1/2} (n+1)$, one obtains (\ref{e re Si A out}). 

\emph{Step 2.} Slightly modifying arguments of Step 1, one can prove that:
\begin{eqnarray}
\re \Si [\A] \supset \left\{ | \alpha | \ge 
\frac{\pi }{(a_2-a_1) b_2^{1/2} } 
\left(\frac12 + \left\lceil  \frac{ b_1^{1/2} }{ b_2^{1/2} -b_1^{1/2} }  - \frac12 \right\rceil \right) 
\right\}
 \text{ if }  \nu_1 \le b_1^{1/2} < b_2^{1/2} < \nu_2 ; \label{e re Si A in} \\
\re \Si [\A] \supset \left\{ | \alpha | > 
\frac{\pi }{(a_2-a_1) b_2^{1/2} } 
\left(\frac12 + \left\lfloor  \frac{ b_1^{1/2} }{ b_2^{1/2} -b_1^{1/2} }  + \frac12 \right\rfloor \right) 
\right\}
 \text{ if }  \nu_1 = b_1^{1/2} < b_2^{1/2} = \nu_2 ; \label{e re Si A ==} \\
\re \Si [\A] \supset \left\{ | \alpha | > 
\frac{\pi }{(a_2-a_1) b_2^{1/2} } 
\right\} \qquad \qquad \qquad
 \text{ if }  0= b_1^{1/2} < b_2^{1/2} \le \nu_1  ; \label{e re Si A b1=0 b2<nu1}
\end{eqnarray}
and that,
\begin{multline}
\text{ in the cases $0<b_1<b_2^{1/2} = \nu_1$ and $\nu_1 < b_1^{1/2} < b_2^{1/2}=\nu_2$}  \\ \text{(\ref{e re Si A out}) and (\ref{e re Si A in}), resp., are valid after the change 
of $| \alpha | \ge \dots $ to $| \alpha | > \dots $}  \label{e re Si A in and out =} .
\end{multline}

\emph{Step 3. Let us show that in the case 
$\nu_1 \le b_1^{1/2} < \nu_2 < b_2^{1/2}$,}
\begin{equation} \label{e re Si A in out}
\re \Si [\A] \supset \left\{ | \alpha | \ge 
\frac{\pi }{(a_2-a_1) b_2^{1/2} } 
\left\lceil  \frac{\frac{1}{2} b_2^{1/2} + b_1^{1/2} }{ b_2^{1/2} -b_1^{1/2} }   \right\rceil
\right\} .
\end{equation}
 Now, the sets $\Sset_n$ covered by $\re \om_n (b)$ have gaps. Indeed, assume that $b_1>0$. Then
\[
\Sset_n = \frac{\pi}{(a_2 -a_1)} \left[ \frac{n}{b_2^{1/2}} , \frac{n}{\nu_2} \right) \cup \frac{\pi}{(a_2 -a_1)} 
\left( \frac{ n+1/2}{\nu_2}  ,  \frac{ n+1/2}{b_1^{1/2}} \right) .
\]
The gap is closed by the sets $\Sset_{n-1}$ and $\Sset_{n+1}$ as soon as 
$(n-1/2) b_1^{-1/2} \ge (n+1) b_2^{-1/2}$. The latter is valid 
if $n \ge \frac{\frac{1}{2} b_2^{1/2} + b_1^{1/2} }{ b_2^{1/2} -b_1^{1/2}} $.
This proves  (\ref{e re Si A in out}) for $b_1 \neq 0$. 
To prove (\ref{e re Si A in out}) in the case $0=b_1^{1/2} $, it is enough to take the limit $b_1 \to 0$.

\emph{Step 4.} Modifying arguments of Step 3, it can be shown that 
\begin{eqnarray}
\re \Si [\A] \supset \left\{ | \alpha | > 
\frac{\pi }{(a_2-a_1) b_2^{1/2} } 
\left(\frac12 + \left\lceil  \frac{ \frac{3}{2} b_1^{1/2} }{ b_2^{1/2} -b_1^{1/2} } \right\rceil \right) 
\right\}
& \text{if} &  0< b_1^{1/2} < \nu_1 < b_2^{1/2} \le \nu_2 ; \label{e re Si A b2 in} \\
\re \Si [\A] \supset \left\{ | \alpha | \ge 
\frac{\pi }{(a_2-a_1) b_2^{1/2} } 
 \left\lceil   \frac{b_2^{1/2} + b_1^{1/2} }{ b_2^{1/2} -b_1^{1/2} }   \right\rceil 
\right\}
& \text{if} & 0 < b_1^{1/2} < \nu_1 \le \nu_2 < b_2^{1/2} ; \label{e re Si A v12 in} \\
\re \Si [\A] \supset \left\{ | \alpha | > 
\frac32 \, \frac{\pi }{(a_2-a_1)  b_2^{1/2} } 
 \right\}
& \text{if} & 0= b_1^{1/2} < \nu_1 < b_2^{1/2} \le \nu_2 ; \label{e re Si A v12 in b1=0} \\
\re \Si [\A] \supset \left\{ | \alpha | \ge 
\frac{2\pi }{(a_2-a_1) b_2^{1/2} }  \right\}
 & \text{if} &  0 = b_1^{1/2} < \nu_1 \le  \nu_2 < b_2^{1/2} . 
\label{e reSiA in b1=0}
\end{eqnarray}

Thus, (\ref{e re Si superset}) follows from the above inclusions.

\subsection{Solvability of bang-bang equations, the proof of Theorem \ref{t uni B} }
\label{ss uni nl}

It is enough to consider the case $ z ^2 \in \overline{\CC_-}$ (see the explanations to (\ref{e Snl sym}) in Section \ref{s Snl pert}).

\begin{lem} \label{l uni B loc}
Let  (\ref{e a b1>0}) be fulfilled. Let $z^2 \in \overline{\CC_-}$. Then there exists
$\ep = \ep (a)>0$ such that problem (\ref{e nonlin z}), (\ref{e ini cond}) has a unique solution $y$ on  $(a-\ep,a+\ep) \cap [a_1,a_2]$.
\end{lem}

\begin{proof}
Firstly, if $b_1=b_2$ in a neighborhood of $a$, the statement is obvious. 

Next, to be specific, assume that $a<a_2$ and that for any $\de>0$ the measure of the set of $x \in (a,a+\de)$ such that $b_1(x) \neq b_2 (x)$ is positive.
Then   (\ref{e a b1>0}) yields
\begin{equation} \label{e bj>0 a}
\text{for any } \de >0, \qquad \meas \{ x \in (a,a+\de) \ : \ b_j (x)>0 \} >0  , \qquad j=1,2.
\end{equation}
It is enough to prove the lemma for the interval $[a, a+\ep)$ with small enough $\ep>0$.

\textbf{Uniqueness.} Assume that there exists at least one solution $y$ to  
 (\ref{e ini cond}), (\ref{e nonlin z}).
Our goal is to prove  that there exist $\ep>0$ and $B_0 (\cdot)$ such that for any solution $y$ to
(\ref{e ini cond}), (\ref{e nonlin z}) the function $\B (y) (\cdot) $ coincides (in $L^1$-sense) with $B_0 (\cdot)$ on $(a,a+\ep)$. This and the structure of equation (\ref{e nonlin z}) immediately implies uniqueness.

\emph{Case 1.} If $c_0=c_1=0$, any solution $y$ to (\ref{e nonlin}), (\ref{e ini cond}) is trivial and $\B (y) \equiv b_1$.

\emph{Case 2.} If $c_0^2 \not \in \RR$ or $\pa_x y^2 (a) = 2 c_0 c_1 \not \in \RR$, then, for a certain nonempty interval $(a,a+\ep)$, the sign of $\im y^2 (x)$ is completely determined by $c_0 and c_1$. This easily gives the desired statement with $B_0$ equal to one of the constraints (either $b_1$, or $b_2$) on $(a,a+\ep)$.

\emph{Case 3.} The case when $c_0=0$ and $c_1^2 \not \in \RR$ is analogous to the Case 2.

\emph{Case 4.}  Assume that $c_0^2 \in \RR \setminus\{0\}$, $c_0 c_1 \in \RR$, and $ z ^2 \in \RR$.
Then either $c_0, c_1 \in \RR$, or $c_0,c_1 \in \ii \RR$.
Moreover, any solution $y$ stays on the same line ($\RR$ or $\ii \RR$) since $ z ^2 \in \RR$.  Therefore $y^2 (x) \in \RR $ and
$\B (y) = b_1$.

\emph{Case 5.}  The case when $c_0 = 0$, $c_1^2 \in \RR \setminus \{0\}$, and $ z ^2 \in \RR$
is analogous to Case 4.

\emph{Case 6.} Assume that $c_0^2 \in \RR \setminus \{0\}$, $c_0 c_1 \in \RR$, and $ z ^2 \in \CC_-$.
 Consider an arbitrary continuous in a vicinity of $a$ branch $\arg_* y (x)$ of $\arg y$. Since $\overline{c_0}c_1 \in \RR$, the arguments of the proof of Lemma \ref{l turn int} (i), (iv) imply
\begin{equation*} \label{e pax arg y}
\pa_x \arg_* y = \frac{\im (\overline{y} \pa_x y) }{|y|^2} \text{ and }
\im [\overline{y(x)} \pa_x y(x)] = -\im  z ^2 \int_a^x
|y(s)|^2 \B (y) (s)   \dd s .
\end{equation*}
It follows from these equalities and (\ref{e bj>0 a}) that $\pa_x \arg_* y (x) >0 $ on $ (a,a+\ep_1(y))$ with a certain $\ep_1 (y)>0$. So for any solution $y$,
\begin{equation} \label{e Byx=}
\B (y) (x) = \left\{ \begin{array}{ll}
b_2 (x), & \text{ if } c_0^2 >0 \\
b_1 (x), & \text{ if } c_0^2 <0
\end{array}  \right. \qquad \text{on } (a,a+\ep_1(y)) .
\end{equation}
This easily implies that it is possible to choose $\ep_1 (y)$ in (\ref{e Byx=}) independently of a solution $y$.

\emph{Case 7.} Assume that $c_0 = 0$, $c_1 \in \RR \setminus \{ 0 \}$, and $ z ^2 \in \CC_-$.
Since
\begin{equation}  \label{e y pr = int y}
y' (x) =c_1 -  z ^2 \int_a^x y(s) \B (y) (s)   \dd s = c_1 -  z ^2 \int_a^x  \int_a^s \B (y) (s) y'(t) \dd t  \dd s,
\end{equation}
we see that
\begin{equation} \label{e im y pr}
\im y' (x) = - \int_a^x  \int_a^s \B (y) (s)  \im [ z ^2y'(t)] \dd t  \dd s .
\end{equation}
Hence, for $x$ slightly greater than $a$,  the complex number $y'(x)$ lies in the same open quadrant as $c_1 - \ii \im  [ z ^2 c_1] $. Due to the equality $y(x) = \int_a^s y'(s) \dd s$, the numbers $y(x)$ lie in the same open quadrant.
So the sign of $c_1$ completely determines $\B (y) (x)$  on $ (a,a+\ep_1(y))$ with a certain $\ep_1 (y)>0$.
As before, it is easy to see that  $\ep_1 (y)$ can be chosen independently of a solution $y$.

\emph{Case 8.} When $c_0 = 0$, $c_1 \in \ii \RR \setminus \{ 0 \}$, and $ z ^2 \in \CC_-$,  one can use arguments similar to Case 7 with the change of (\ref{e im y pr}) to
\begin{equation*} %\label{e }
\re y' (x) = - \int_a^x  \int_a^s \B (y) (s)  \re [ z ^2y'(t)] \dd t  \dd s .
\end{equation*}
 This completes the proof of uniqueness.

\textbf{Existence.} The above arguments also provide an algorithm for assigning $B_0 (x)$ the values of $b_1 (x) $, or $b_2 (x)$ for $x>a$ such that there exists a solution $y$ to $y'' = -  z ^2 B_0 y$ on a certain interval $(a,a+\ep)$ satisfying the initial conditions (\ref{e ini cond}) and additionally $B_0 (x) = \B (y) (x)$ on $(a,a+\ep)$. Then $y$ is a solution to (\ref{e nonlin}), (\ref{e ini cond}) on $(a,a+\ep)$. This completes the proof.
\end{proof}

Assume that (\ref{e nonlin}), (\ref{e ini cond}) has a unique solution $y$ on  $(x_1, x_2)$ with $x_1<a<x_2 < +\infty$.

\begin{lem}  \label{e lim y y pr}
There exist $\lim_{x \to x_2-} y (x)$ and $\lim_{x \to x_2-} y' (x)$.
\end{lem}

\begin{proof}
The statement follows from the Bellman-Gronwall lemma, from the integral reformulation of problem (\ref{e nonlin z}), (\ref{e ini cond})  given by the equation
\begin{eqnarray}
 y (x) & = & c_0 + c_1 (x-a) - z^2 \ \int_{a}^x (x-s) \ y (s) \B (y) (s) \dd s ,
\label{e int nonlin}
\end{eqnarray}
and from the fact that $|\B (y) (\cdot)| \le b_2 (\cdot) \in L^1_\RR (x_1,x_2)$ for arbitrary $y$.
\end{proof}

This lemma, the equality $y'(x) =  c_1 - z ^2 \int_a^x y(s) \B (y) (s)  \dd s$, 
and (\ref{e int nonlin})  imply that $y$ and $y'$ can be considered 
as absolutely continuous functions on $[a,x_2]$. 
This arguments allow us to extend $y$ on $[a,a_2]$. 
The resulting solution is unique due to Lemma \ref{l uni B loc}. 
This completes the proof of  Theorem \ref{t uni B}.

\section*{Acknowledgements} 
IK and IV are very grateful to J{\"u}rgen Prestin  for the hospitality of 
the University of L{\"u}beck and were partially supported by the project 
''EU-Ukrainian Mathematicians for Life Sciences'' 
(Marie Curie Actions - International Research Staff Exchange Scheme (IRSES) FP7-People-2011-IRSES, project number 295164). 
IK  was partially supported by the project no.15-1vv{\textbackslash}19 
"Metric spaces, harmonic analysis of functions and operators, singular and nonclassical problems for differential equations" 
of the Faculty of Mathematics and Information Technology at Donetsk National University. 

The authors would like to thank Herbert Koch for the question how optimal structures look like, and 
Michael Weinstein for the question on bands of admissible frequencies.  
Sections \ref{ss period defect} and \ref{ss ex A} were aimed to address these questions. 
IK is very grateful to Rostyslav Hryniv,  ‎Evgenii Yakovlevich Khruslov, and Sergiy Maksymenko for their interest to this research and for opportunities to give talks in the seminars of their departments.

\end{document}